\documentclass[11pt]{amsart}

\usepackage{graphicx}
\usepackage{amssymb,amsmath,amsthm,newlfont,enumerate}
\usepackage{xcolor}
\usepackage{longtable}
\usepackage{hyperref}
\usepackage{microtype}
\usepackage[ruled,vlined]{algorithm2e}
\usepackage{multirow}
\usepackage{listings}
\usepackage{multicol}
\usepackage{bbm, dsfont}
\usepackage{etoolbox}
\usepackage[caption = false]{subfig}
\usepackage{comment}
\usepackage{natbib, float}
\usepackage{placeins}

\theoremstyle{plain}
\newtheorem{theorem}{Theorem}[section]

\newtheorem{lemma}[theorem]{Lemma}
\newtheorem{proposition}[theorem]{Proposition}

\theoremstyle{definition}

\theoremstyle{remark}
\newtheorem*{remark}{Remark}
\newtheorem*{remarks}{Remarks}

\newcommand{\X}{{\mathbf X}}

\newcommand{\RR}{{\mathbb R}}

\newcommand{\ZZ}{{\mathbb Z}}
\newcommand{\cA}{{\cal A}}

\newcommand{\cL}{{\cal L}}

\DeclareMathOperator{\card}{card}

\DeclareMathOperator{\diam}{diam}

\renewcommand{\hat}{\widehat}
\renewcommand{\tilde}{\widetilde}

\title[A quantitative Heppes Theorem]{A quantitative Heppes Theorem and multivariate Bernoulli distributions}

\author[Fraiman]{Ricardo Fraiman}
\address{Centro de Matem\'atica, Facultad de Ciencias, Universidad de la Rep\'ublica, Montevideo, Uruguay.}
\email{rfraiman@cmat.edu.uy} 

\author[Moreno]{Leonardo Moreno}
\address{Instituto de Estad\'{i}stica, Departamento de M\'etodos Cuantitativos, FCEA, Universidad de la Rep\'ublica, 
Montevideo, Uruguay.}
\email{leonardo.moreno@fcea.edu.uy} 

\author[Ransford]{Thomas Ransford}
\address{D\'epartement de math\'ematiques et de statistique, Universit\'e Laval,
Qu\'ebec City,  Canada. }
\email{ransford@mat.ulaval.ca}

\begin{document}
\begin{abstract}
Using some extensions of a theorem of Heppes  on finitely supported discrete probability measures,
we address the problems of classification and testing based on projections. 
In particular, when the support of the distributions is known in advance (as for instance for multivariate Bernoulli distributions),
a single suitably chosen projection determines the distribution. Several applications of these results are considered.
\end{abstract}

\keywords{Classification, Discrete tomography, Heppes theorem, Multivariate Bernoulli,  Random projections,  Testing hypothesis}

\makeatletter
\@namedef{subjclassname@2020}{\textup{2020} Mathematics Subject Classification}
\makeatother

\subjclass[2020]{60E05, 62E10, 62G10, 62H30, 62H15}

\maketitle

\section{Introduction}
The problem of dimension reduction is of increasing importance in modern data analysis. In this setting, we  address the problem of when we can efficiently perform some important statistical analysis on high-dimensional discrete data, such as learning or  hypothesis testing.

For supervised  classification problems with high-dimensional binary data, nearest-neighbour or kernel based rules have a poor behaviour, and some other well-known techniques like  Random Forest (\cite{breiman2001}) or  Support Vector Machines (SVM see \cite{boser1992}) are alternatives used in practice.  However, as mentioned in  \cite{xu2021}, there are just a few procedures particularly well designed for discrete data (see \cite{needell2018,molitor2019,molitor2021}). In particular, \cite{needell2018} consider a classification rule for binary data,  where the classification rule is updated recursively using projections, but over  hyperplanes determined by subsets of the original coordinates  chosen at random.

In this article, we propose an alternative approach that provides competitive procedures for the two main problems under consideration, hypothesis testing and learning.
It also has potential applications to other statistical problems, since it provides a general dimension-reduction technique, based on a bound for the total variation distance between probability measures in terms of the distances between certain of their lower-dimensional projections.
For example, this idea can be applied to perform clustering or to use depth functions in lower dimensions, among other applications that are not considered in the present manuscript.

Our approach is based on a variant of a theorem of Cram\'er and Wold (\cite{cramer36}).
According to the original Cram\'er--Wold  theorem, if $P$ and $Q$ are Borel probability measures on $\RR^d$
whose projections $P_L,Q_L$ onto each line $L$ satisfy $P_L=Q_L$, then $P=Q$. 
There are several extensions of this theorem,
in which one assumes more about the nature of the measures $P,Q$ 
and less about the set of lines $L$ for which $P_L=Q_L$.
For example, if $P$ and $Q$ have  moment generating functions that are finite in a neighbourhood of the origin,
and if $P_L=Q_L$ for all lines $L$ in a set of positive measure (in a natural sense),
then $P=Q$. 
Articles on this subject include \cite{renyi1952},
\cite{gilbert1955},  \cite{belisle1997} and  \cite{cuesta2007}.

If  $P$ and $Q$ are supported on finite sets, then it is even possible
to differentiate between them using only a finite set of projections.
The following result is due to Heppes  \cite[Theorem~$1'$]{heppes1956}.
We denote by $P_H$  the projection of a probability measure $P$ onto a general subspace $H$.

\begin{theorem}\label{T:Heppes56}
Let $Q$ be a discrete probability measure on $\RR^d$ whose support consists of $k$  points.
Let $H_1, H_2, \ldots, H_{k+1}$ be subspaces of $\RR^d$ respectively of dimensions $m_1, \ldots, m_{k+1}$, such that no two of these subspaces are contained in a single hyperplane, that is, no arbitrary straight line in $\RR^d$ can be perpendicular to more than one of the $H_i$. If $P$ is a Borel probability measure on $\RR^d$  such that 
$P_{H_i}= Q_{H_i}$, $i=1, \ldots, k+1$, then $P=Q$. 
\end{theorem}

An example of \cite{renyi1952} shows that a minimum of $(k+1)$ subspaces are needed in Theorem~\ref{T:Heppes56}. On the other hand, the ambient dimension $d$ plays no role in this result, which could just as easily have been formulated in any separable Hilbert space. We also note that, though $Q$ is assumed discrete, no such assumption is made about $P$. This makes the result suitable for hypothesis testing.

We establish a quantitative version of Heppes' theorem (see Theorem \ref{teo1} in next section),
where we show that the total variation distance between $P$ and $Q$ is at most the
sum of the total variation distances between their projections.
We also prove  a refinement of Heppes' theorem where, if we know in advance the support of the distribution, then merely one suitably chosen projection suffices to determine the probability measure (see Proposition~\ref{unsoloH} in the next section). This projection just needs to be chosen to avoid a finite or countable number of bad directions, and thus almost all directions will be adequate. In particular, this is the case of multivariate binary distributions, which will be the focus of part of this paper.

The main results regarding projections are presented in Section~\ref{quantitative}. In Section~\ref{learning} we propose a procedure for learning from projections, and adapt it for the learning problem for discrete tomography. Next we consider the case of testing for multivariate binary data in Section~\ref{binary} and analyze the testing problem for some well-known distributions for multivariate binary data, an important problem for many applications (see for instance the nice review on this problem in \cite{balakrishnan2018}).
 Section~\ref{unasola} addresses a different problem, namely  if we can still say  something when we have only one realization $\mathbf X$ of a multivariate Bernoulli distribution. 

Our results are illustrated with some simulation examples in Section~\ref{simus} and three real-data examples
in Section~\ref{S:realdata}.

We end this section with a remark about possible extensions of our ideas to continuous distributions. As mentioned above, the starting point of this article is  Heppes' theorem, which allows us to differentiate between two finitely supported measures on $\RR^d$ using only finitely many projections. For more general measures, it is usually necessary to use
an infinite set of projections (see e.g.\ \cite{hamedani75}).
However, it turns  out that finitely many projections suffice if both $P$ and $Q$ are elliptical distributions
(a class of measures that includes not only Gaussian distributions but also many other continuous multivariate distributions of interest, including some with infinite first moments). This result, which may be viewed as a kind of analogue of Heppes' theorem for continuous distributions,
is discussed in more detail at the end of the paper.
%%%%%%%%%%%%%%%%%%%%%%%%%%%%%%%%%%%%%%%%%%%%%%

\section{Some variants of Heppes' Theorem} \label{quantitative}

Given two probability measures $P,Q$ on a measurable
space $(E,\cA)$, we consider the \emph{total variation distance} between them, namely
\[
d_{TV}(P,Q):=\sup_{A\in\cA}|P(A)-Q(A)|.
\]

Also, given a Borel probability measure $P$ on $\RR^d$ and a subspace $H$ of $\RR^d$,
we write $P_H$ for the projection of $P$ onto $H$, namely the Borel probability measure on $H$ given by
\[
P_H(B):=P(\pi_H^{-1}(B)),
\]
where $\pi_H:\RR^d\to H$ is the orthogonal projection of $\RR^d$ onto $H$.

Clearly, if $P,Q$ are Borel probability measures on $\RR^d$ and $H$ is a subspace of $\RR^d$,
then
\[
 d_{TV}(P_H,Q_H)\le d_{TV}(P,Q).
\]
In this section we investigate some inequalities going in the opposite sense.

\subsection{A quantitative Heppes theorem}\

The following result is a quantitative generalization of Theorem~{\ref{T:Heppes56}.

\begin{theorem}\label{teo1}
Let $Q$ be a discrete probability measure on $\RR^d$
whose support contains at most  $k$ points.
	Let $H_1,\dots,H_{k+1}$ be  subspaces of $\RR^d$ such that 
	$H_i^\perp\cap H_j^\perp=\{0\}$ whenever $i\ne j$.
Then, for every Borel probability measure $P$ on $\RR^d$, we have
	\begin{equation}\label{E:quantHeppes}
	d_{TV}(P,Q)\le \sum_{j=1}^{k+1} d_{TV}(P_{H_j},Q_{H_j}) .
	\end{equation}
\end{theorem}

The main idea of the proof is embodied in the following lemma,
which is a useful result in its own right.
	
	\begin{lemma}\label{L:teo1}
	Let $P,Q$ and $H_1,\dots,H_{k+1}$ be as in the statement of the theorem. 
	Then
	\begin{equation}\label{E:pmayorq}
	P(\{x\})-Q(\{x\})\le \max_{1\le j\le k+1} \Bigl(P_{H_j}(\{\pi_{H_j}(x)\})-Q_{H_j}(\{\pi_{H_j}(x)\})\Bigr) \quad(x\in \RR^d).
	\end{equation}
	\end{lemma}

\begin{proof}
	Let $x\in \RR^d$. Since the $(k+1)$ sets $(x+H_j^{\perp})\setminus\{x\}, ~(j=1,2,\dots,k+1)$ are pairwise disjoint,
	and since $Q$ is supported on a set containing at most $k$ points, 
	there exists a $j$ such that $(x+H_j^{\perp})\setminus\{x\}$ is disjoint from the support of $Q$.
	Hence
\[	
Q(x+H_j^\perp)=Q(\{x\}).
\]
We then have
	\begin{align*}
	P(\{x\})-Q(\{x\})
	&=P(\{x\})- Q(x+H_i^{\perp})\\
	&\le P(x+H_j^{\perp})- Q(x+H_j^{\perp})\\
	&=P_{H_j}(\{\pi_{H_j}(x)\})-Q_{H_j}(\{\pi_{H_j}(x)\}).
	\end{align*}
	This establishes \eqref{E:pmayorq}.

\end{proof}

\begin{proof}[Proof of Theorem~\ref{teo1}]
For each $j\in\{1,2,\dots,k+1\}$, let $A_j$ be the set of $x\in\RR^d$ such that
$(x+H_j^{\perp})\setminus\{x\}$ is disjoint from the support of $Q$.
The proof of  Lemma~\ref{L:teo1} shows that $\cup_{j=1}^{k+1}A_j=\RR^d$.
Also, if $B_j$ is a Borel subset of $A_j$, then $(B_j+H_j^\perp)\setminus B_j$ is contained in
$\cup_{x\in A_j}\bigl((x+H_j^\perp)\setminus\{x\}\bigr)$, which is disjoint from the support of $Q$,  and hence
\[
Q(B_j+H_j^\perp)=Q(B_j).
\]

Now let $P$ be a Borel probability measure on $\RR^d$,
and let $B$ be an arbitrary Borel subset of $\RR^d$.
Set $B_1:=B\cap A_1$ and $B_j:=B\cap (A_j\setminus A_{j-1})$ for $j=2,\dots,k+1$. 
Then $B_1,\dots,B_{k+1}$ is a Borel partition of $B$ such that $B_j\subset A_j$ for all $j$. Hence
\begin{align*}
P(B)-Q(B)
&=\sum_{j=1}^{k+1}\Bigl(P(B_j)-Q(B_j)\Bigr)
=\sum_{j=1}^{k+1}\Bigl(P(B_j)-Q(B_j+H_j^\perp)\Bigr)\\
&\le\sum_{j=1}^{k+1}\Bigl(P(B_j+H_j^\perp)-Q(B_j+H_j^\perp)\Bigr)\\
&=\sum_{j=1}^{k+1}\Bigl( P_{H_j}(\pi_{H_j}(B_j))-Q_{H_j}(\pi_{H_j}(B_j))\Bigr)\\
&\le \sum_{j=1}^{k+1}d_{TV}(P_{H_j},Q_{H_j}).
\end{align*}
Also, since $P,Q$ are both probability measures, we have
\[
Q(B)-P(B)=P(\RR^d\setminus B)-Q(\RR^d\setminus B)\le  \sum_{j=1}^{k+1}d_{TV}(P_{H_j},Q_{H_j}).
\]
Finally, combining  the last two inequalities, we obtain \eqref{E:quantHeppes}.

\end{proof}

%%%%%%%%%%%%%%%%%%%%%%%
\subsection{A Heppes-type result for distributions with pre-determined support}\

If we know in advance the location of the supports of $P$ and $Q$, rather than merely their cardinality,
then it suffices to use just one judiciously chosen subspace $H$. 
This is even true if $P,Q$ have countable supports, 
as the following result shows.

\begin{proposition}\label{unsoloH}
	Let $P,Q$ be probability measures on $\RR^d$ supported on a countable set $E\subset\RR^d$.
	Let $H$ be a subspace of $\RR^d$ such that $H^\perp\cap(E-E)=\{0\}$.
	Then 
	\[
	d_{TV}(P,Q)=d_{TV}(P_H,Q_H).
	\] 
	In particular, $P_H=Q_H$ implies $P=Q$.
\end{proposition}

\begin{remark}
We can even choose $H$ to have dimension equal to $1$. Indeed,
this amounts to ensuring that the hyperplane $H^\perp$ misses the countable set $E-E$.
\end{remark}

\begin{lemma} 
	Let $E$ be a  countable subset of $\RR^d$
	and let $H$ be a subspace of $\RR^d$ with the property that $H^\perp\cap(E-E)=\{0\}$. 
	Then the restriction of  $\pi_H$ to $E$ is injective.
\end{lemma}

\begin{proof}
	If $x,y\in E$  and $\pi_H(x)=\pi_H(y)$, then $x-y\in (E-E)\cap H^\perp=\{0\}$, so $x=y$.
\end{proof}

\begin{proof}[Proof of Proposition~\ref{unsoloH}]
Let $B$ be a Borel subset of $\RR^d$. 
From the lemma, we have $E\cap B=E\cap \pi_H^{-1}\pi_H(E\cap B)$, so
\begin{align*}
P(B)&=P(E\cap B)=P(E\cap \pi_H^{-1}\pi_H(E\cap B))\\
&=P(\pi_H^{-1}\pi_H(E\cap B))=P_H(\pi_H(E\cap B)),
\end{align*}
and likewise for $Q$. Hence
\[
|P(B)-Q(B)|=|P_H(\pi_H(E\cap B))-Q_H(\pi_H(E\cap B))|\le d_{TV}(P_H,Q_H).
\]
It follows that $d_{TV}(P,Q)\le d_{TV}(P_H,Q_H)$. The reverse inequality is obvious,
so we have equality.
\end{proof}
 
\textbf{A simple illustration of Proposition \ref{unsoloH}}. 
Suppose that we have a random vector $(X_1,X_2) \in \mathbb{R}^2$, 
with $X_1 \sim \textrm{Ber}(p_1)$ and $X_2 \sim \textrm{Ber}(p_2)$. 
In this case, $E=\{(0,0), (1,0),(0,1),(1,1)\}$. 
We consider one-dimensional subspaces $H$, 
and we look at the orthogonal projections of the points in the support $E$ on $H$, see Figure~\ref{fig1}.
There are only $4$ subspaces (in red in Figure~\ref{fig1}) that do not fulfill the condition  $H^\perp\cap(E-E)=\{0\}$, 
namely those given by  $\{(x,y) \in \mathbb{R}^2, \  ax+by=0\}$ with $(a,b) \in \{(1,0),(1,1), (0,1), (1,-1)\}$.

\begin{figure}[htb]  
\centering
\includegraphics[width=45mm]{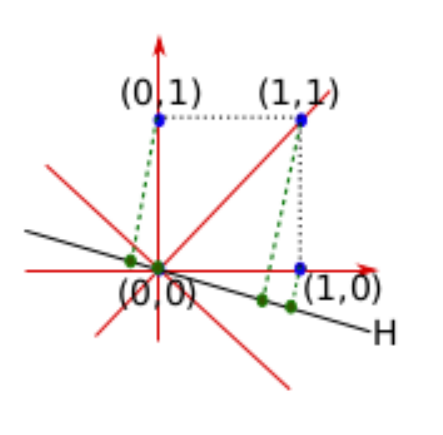}
\caption{Illustration of Proposition \ref{unsoloH}.} \label{fig1}  
\end{figure}

\subsection{Relation with other metrics.}\

Theorem~\ref{teo1} and  Proposition~\ref{unsoloH} are expressed in terms
of the total variation metric $d_{TV}$. For discrete distributions $P,Q$ on $\RR^d$, 
this metric is also given by
\begin{align*}
d_{TV}(P,Q)&= \frac{1}{2} \sum_x \Bigl| P(\{x\})- Q(\{x\})\Bigr|\\
&= \inf \Bigl\{\mathbb P( X \neq Y) : (X,Y)  \text{~such that~} X \sim P\text{~and~}Y \sim Q\Bigr\}.
\end{align*}

There are several other commonly used metrics on the space of probability measures;
see \cite{gibbs2002} for an account of these.
We mention one that will be used in what follows.

The \emph{Wasserstein--Kantorovich metric} of order $p \geq 1$,
defined for distributions on $\RR^d$ belonging to $L^p$,  is given by
\[
d_{W,p}(P,Q):=\inf\Bigl\{ \mathbb{E}(|X-Y|^p)^{1/p}:  (X,Y)  \text{~such that~} X\sim P \text{~and~}Y\sim Q\Bigr\}.
\]
For one-dimensional distributions, an equivalent formulation is given by
\begin{equation} \label{Wass}
d_{W,p}(P,Q) = \left(\int_0^1 \vert F^{-1}(t)- G^{-1}(t)\vert ^p \right)^{1/p},
\end{equation}
where $F$ and $G$ stands for the distribution functions of $P$ and $Q$.

From \cite[Theorem~4]{gibbs2002}, 
if $P,Q$ are both supported on a bounded set $E\subset\RR^d$ of diameter $\diam(E)$, then
\[
d_{W,1}(P,Q)\le \diam(E) d_{TV}(P,Q),
\]
and, if $E$ is finite and $d_{\min}(E)$ denotes the minimum separation between points of $E$, then
\begin{equation}\label{TVW}
d_{W,1}(P,Q)\ge d_{\min}(E) d_{TV}(P,Q).
\end{equation}
It follows that, if $(P_n)$ is a sequence of probability distributions all supported on the same finite set,
then $P_n$ converges to $P$ weakly if and only if $d_{TV}(P_n,P)\to0$.

For the case of discrete tomography problem we will consider  Mallow's $L^2$-distance  between histograms \cite{mallows1972}, based on the Wasserstein distance given in (\ref{Wass}).} For histograms, the article \cite{irpino2006} provides an exact calculation of this distance by an efficient algorithm.

\section{ Classification based on projections and discrete tomography }\label{learning}

Discrete tomography, a term introduced by Larry Shepp in 1994, focuses on the reconstruction of binary (black and white) images based on a finite set of projections of the data, see for instance \cite{gardner96bb}. 

Let $E\subset \ZZ^d$ be the domain of the binary images, let $N:=\card(E)$,
and let $F$ be a subset of $E$. In this setting we only have access to the information given on parallel lines orthogonal to certain directions $u$ of the images. More precisely, given a direction $u$ in the unit sphere, the $X$-ray of the set $F$ in direction $u$ provides the number of points in the set on lines parallel to $u$, that is, if $L_u$ stands for the line through the origin parallel to $u$, we are given the function
\begin{equation} \label{tomography}
T_u(F)(x):= \card(F \cap (x+L_u)), \ \ x \in \RR.
\end{equation}

Given a direction $u$, we estimate $T_u(F)(x)$ based on an histogram built up from a partition of $\RR$  $B_1,\ldots,B_m$ as follows:  for $x \in B_j$,  define
\[
\hat{T}_u(F)(x):= \card \left \{  \cup_{y\in B_j}(F \cap (y+L_u))  \right \}.
\]
So we have the histogram of the points determined by the binary image of the points in the direction $u$, see Figure \ref{fig2}.
 
\begin{figure}[htb]  
\centering
\includegraphics[width=65mm]{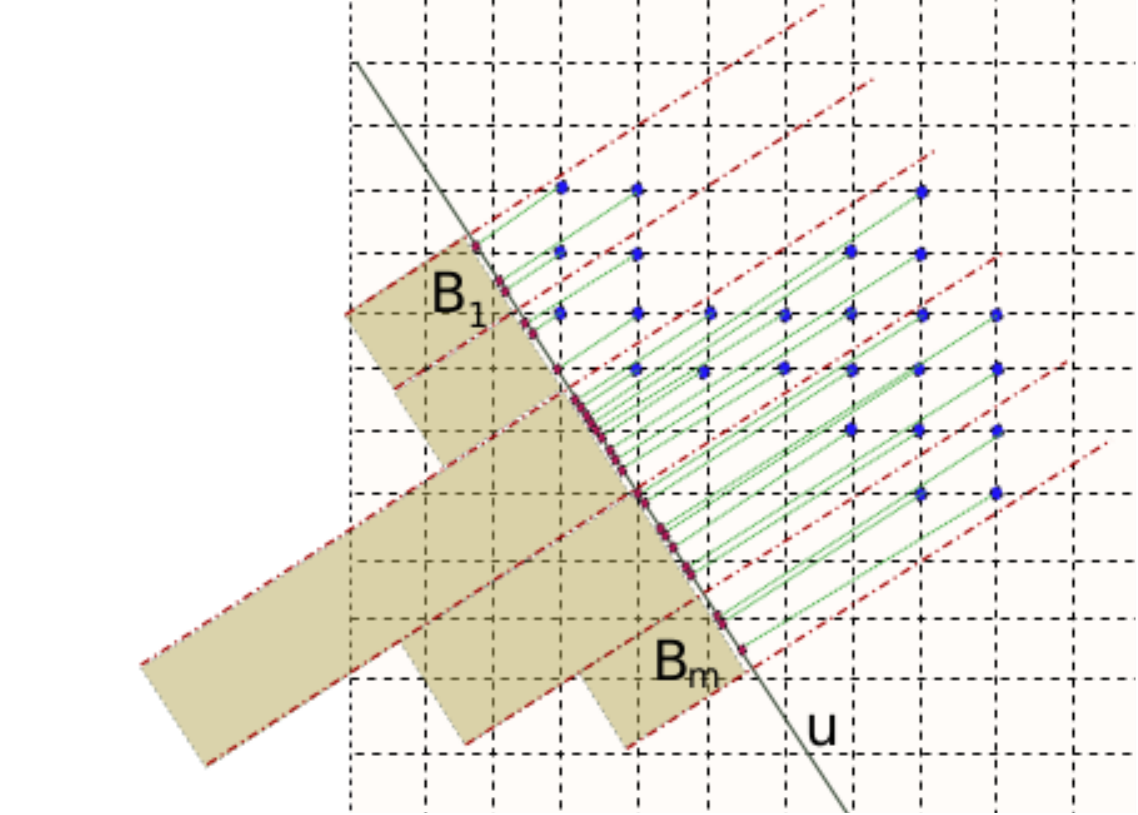}
\caption{Representation of the histogram of the projection of the data in the direction $u$.} \label{fig2}  
\end{figure}

\subsection{An algorithm with random projections}

First we consider the following general classification problem into $m$ classes: given an iid training sample, 
$\mathcal{D}^n=(\mathbf X_n,\mathbf Y_n)=\hspace{-0.1cm}\{(X_1,Y_1),\dots,(X_n,Y_n)\}$  
with the same distribution as  $(X,Y)\in E\times \{1, \ldots, m\}$, we want to classify a new data $X_{n+1}$ for which the label $Y_{n+1}$ is unknown.
Let $P_l, \ l=1, \ldots, m$ stand for the unknown discrete distributions of $X \vert Y=l$  and let $P_{nl}$ be the respective empirical distributions based on the training sample.

Let  $\eta(x)$ denote the conditional mean of $Y$ given $X=x$; namely, $\eta(x)= \mathbb{E}(Y|X=x)$. It is well known that the optimal rule for classifying a single new datum $X$ is given by the Bayes' rule, which, for binary classification, reduces to $g^*(X)=\mathbb{I}_{\{\eta(X)\geq 1/2\}}$. 
Bayes' rule for finitely supported distributions to classify $x$ corresponds to defining
\begin{equation}
\begin{aligned}
g^*(x) = \tilde Y :&= \arg \max_{l=1,\ldots,m} \frac{P(X=x \vert Y=l)P(Y=l)}{P(X=x)} \\
&= \arg \max_{l=1,\ldots,m} \frac{P_l(X=x)P(Y=l)}{P_n(X=x)},
\end{aligned}
\end{equation}
and a plug-in rule is given by
\begin{equation} \label{plugin}
g_n(X_{n+1}, \mathcal D_n)= \hat Y_{n+1}  := \arg \max_{l=1,\ldots,m} \frac{P_{nl}(X=x)P_n(Y=l)}{P_n(X=x)},
\end{equation}
which, for the binary classification case, reduces to
\begin{equation}\label{binary1}
\begin{aligned} 
g^*(x)&=\mathcal I_{\{P_1(X=x)P(Y=1)> P_0(X=x)P(Y=0)\}}; \\
 g_n(x)&=\mathcal I_{\{P_{n1}(X=x)P_n(Y=1)> P_{n0}(X=x)P_n(Y=0)\}}.
\end{aligned}
\end{equation}

Since we only have access to the projections of the distributions, our first approach is to make use of the relation (\ref{E:pmayorq}).
Let $\mathcal D^{nl}_j, \ l=1, \ldots m, \ j=1, \ldots k+1$ be the projections of the training sample with label $l$ on the subspace $H_j$, and let $X_{n+1,j}$ be the projection of a new data on $H_j$. To simplify the notation, we will consider the binary classification problem, $m=2$.

A possible  classification procedure will be to find under which member of the family of empirical distributions $P_{nj}, \ j =1, \ldots, m$ the new data $X_{n+1}$ ``fits better".  For that purpose, we  consider the distributions $\tilde P_{j(n+1)}, \ j=1, \ldots, m$ defined by just adding the data $X_{n+1}= x$ to each one of the $j$ empirical distributions. Next we consider the vector of distances $d_{TV}(P_{nj}, \tilde P_{(n+1)j}), \ j=1, \ldots m$ and we classify the point $X_{n+1}$ into the class $j$ for which the distance is minimal. We will use this last approach in what follows. Consistency is shown in the next theorem, where, for notational simplicity, we consider the binary classification problem, where now $\mathbb E(Y \vert X=x)= P(Y=1 \vert X=x)$. Let $L^* = P(g^*(X)\neq Y)$ be the Bayes risk, i.e.\ the risk of the best classifier, and let $L_n= L(g_n)= P(g_n(X, \mathcal D_ n \neq Y \vert \mathcal D_n))$. 
Recall that the classifier is weakly consistent if
\[
\mathbb E(L_n)= P(g_n(X, \mathcal D_n)\neq Y) \to L^*, \ \mbox{as} \ n \to \infty,
\]
and strongly consistent if
\[
\lim_{n \to \infty} L_n = L^* \ \ \mbox{with probability} \  1.
\]
See for instance \cite{devroye1996} for further details. Since (\ref{TVW}) implies that 
$d_{TV}(P_{n0},P_0)$ and $d_{TV}(P_{n1}, P_1)$ converge to zero a.s.,  
it is clear that the plug-in classification rule given by equation (\ref{binary1}) is strongly consistent.

Next we will prove the consistency of the new proposal rule given by:
\begin{equation} \label{reglanueva} 
g_n(x_{n+1}):= 
\begin{cases}
1-i, \ i=0,1, &\text{if~} P_{ni}(x_{n+1})=0,\\
\mathcal I_{\{d_{TV}(P_{n1}, \tilde P_{(n+1)1}) < d_{TV}(P_{n0}, \tilde P_{(n+1)0})\}}, &\text{otherwise}.
\end{cases}
\end{equation}

Observe that the new rule (\ref{reglanueva}) does not require us to have any notion of distance in the space of covariates $E$.

\begin{theorem}\label{teocon}
Let	 $\mathcal{D}^n=(\mathbf X_n,\mathbf Y_n)=\hspace{-0.1cm}\{(X_1,Y_1),\dots,(X_n,Y_n)\}$  
be an iid training sample with the same distribution as  $(X,Y)\in E\times \{0, 1\}$. 	Then the classification rule given by equation (\ref{reglanueva}) is strongly consistent provided that the training sample satisfies the condition that $Y(x)=1$ for all $x$ such that $P_1(x)> P_0(x)$. 
\end{theorem}
\begin{proof}
	First observe that since the support is finite, if $\tilde P_{n0}(x_{n+1})=0$ for $n$ large enough, then $P_0(x_{n+1})=0$, so we can replace  $P_{ni}, \ i=0,1$ by $P_{i}$ in the top line of (\ref{reglanueva}).
Next observe that it suffices to show that, if 
	\begin{equation} \label{hip}
	\mathcal I_{\{P_1(X=x)> P_0(X=x)\}} =1,
	\end{equation}
	then $\lim_{n \to \infty} \mathcal I_{\{d_{TV}(P_{n1}, \tilde P_{(n+1)1}) < d_{TV}(P_{n0}, \tilde P_{(n+1)0})\}} = 1 \ \mbox{with probability 1}$. 
	
	Since $d_{TV}(P_{n0}, P_0)$ and $d_{TV}(P_{n1}, P_1)$ converge to 0 a.s., the problem reduces to showing that equation (\ref{hip}) implies that 
	\[
	\lim_{n \to \infty} \mathcal I_{\{d_{TV}(P_1, \tilde P_{(n+1)1}) < d_{TV}(P_0, \tilde P_{(n+1)0})\}} = 1
	\]
	with probability 1.
	
	We only have to consider the case where    $P_0(x_{n+1})\neq 0$. We have $d_{TV}(P_1, \tilde P_{(n+1)1} )\to 0$ a.s., while, with probability one,
\[
d_{TV}(P_0, \tilde P_{(n+1)0}) >  \frac{1}{2}\Bigl| P_0(x_{n+1})- \frac{1}{(n_0+1)} \Bigr| > P_0(x_{n+1})/4
\]
for all $n$ large enough. 
\end{proof}

Since we only have access to the projections of the distributions, we make use of the bound based on projections given in the inequality (\ref{E:quantHeppes}) and replace the distances by their bounds. 

Let $\mathcal D^{nl}_j, \ l=1, \ldots m, \ j=1, \ldots k+1$ be the projections of the training sample with label $l$ on the subspace $H_j$ and let $X_{n+1,j}$ be the projection of a new datum on $H_j$.

\begin{center}
	\begin{algorithm}[H]
	\label{alg1}
		\SetAlgoLined
		Input $H_1, \ldots, H_{k+1}$ ; $\mathcal D^{nl}_j, \ l=1, \ldots m, \ j=1, \ldots k+1$ and $X_{n+1,j} \ j=1 \ldots k+1$ \;
		For each pair $(j,l)$, calculate the empirical distribution of $\mathcal D^{nl}_j$, $P_{\mathcal D^{nl}_j}$  \;
		Add the new projected data $X_{n+1,j}$ to each pair $(j,l)$ of distributions $P_{\mathcal D^{nl}_j}$ and obtain $P_{\mathcal D^{(n+1)l}_j}$ \;
		
			Calculate  $d_b(l):= \max_{j=1, \ldots, k+1}d(P_{\mathcal D^{nl}_j},P_{\mathcal D^{(n+1)l}_j} ), $ for $l=1, \dots m$,
			where $d$ is a distance between the empirical distributions in $\mathbb{R}$ ;
	
		Output $Y_{n+1}=l_s = \arg\min_{l=1, \ldots m}d_b(l)$.
		
		\caption{Classification algorithm for learning from projections.}
	\end{algorithm}            
\end{center}

\begin{remark} In recent years, there has been a substantial literature on the subject of robustness for learning algorithms focussing on adversarial perturbations of the data (adversarial attacks), see for instance \cite{bertsimas2019,lopez2022}, and the references therein. It treats the question of how changing a small fraction of  data at arbitrary positions affects the learning strategies. Although distinct, this problem is related to the classical notion of breakdown point in robustness.

Among the kinds of perturbations considered,
there are three main different settings: robustness against uncertainty in covariate features  only in the testing sample, robustness against uncertainty only in labels, and robustness against uncertainty in covariate features and in the training sample, the last of these being the wildest. Although this problem   is quite far from the central objective of the present article, it is not difficult to show that algorithm (\ref{reglanueva}) is robust with respect to these kinds  of perturbations.  Indeed, from the proof of  Theorem \ref{teocon}, one can derive that the procedure will be robust at $x$ as long as the fraction $\alpha$ of adversarial perturbation points is smaller than $|P(x)-Q(x)|$.  This condition seems to be mandatory for robustness for any classification rule from discrete data.
\end{remark}

\subsection{Discrete tomography: classification without reconstruction}

In discrete tomography, for each $F$ we only have access to the  values $ T_u(F)(x_i)$ for each direction $u$ on a given grid $\mathbf x=\{x_1, \ldots, x_M\}$.  Once again, for notational simplicity, we  consider the case of binary classification. Let $ \mathcal D^n = \{\mathbb F_n, \mathbb Y_n\}=\{(F_1, Y_1), \ldots, (F_n, Y_n)\}$ (the training sample) with the same distribution as $(F,Y) \in E \times \{0,1\}$. Let $n_1=\sum_{i=1}^n Y_i$ and $n_0 = n-n_1$ be the sizes of each class in the training sample. Denote by $F_{ij}, \ j=1, \ldots n_i, \ i=0,1$, the data at each class according to its label $i$. We want to classify a new data $F$.
	Let   $\mathbf T_u(F_{0j}):= ( T_u(F_{0j})(x_1), \ldots, T_u(F_{0j})(x_M)), \ j=1, \ldots n_0$, and  $\mathbf T_u(F_{1j}):= ( T_u(F_{1j})(x_1), \ldots, T_u(F_{1j})(x_M)), \ j=1, \ldots n_1$, corresponding to direction $u$.	
Then define $\hat{\mathbf T}_u(F_{ij})$  as the histogram associated to the points of $F_{ij}$ projected on the direction $u$. So, to each discrete tomography and direction, we associate the corresponding histogram.

We will denote by $d$ the Mallow $L^2$-distance between histograms. Next, based on this distance, for each direction $u$ we apply the classical $k$-nearest neighbour classification rule and then combine the results for the different directions $u$ to produce the output.
 More precisely, we consider the following algorithm:

\begin{center}
\begin{algorithm}[H] \label{alg4}
\SetAlgoLined
Input directions $u_1, \ldots, u_{k+1}$ \; 
Input ``projected" empirical histograms  $\hat{\mathbf T}_u(F_{ij})$  with  $j=1, \ldots n_i, \, i=0,1$ for each direction $u$ \;
Calculate the histogram corresponding to the new data $F_{n+1}$\;
Let $d$ be a distance between histograms, given the direction $u$, determine $d\left( \hat{\mathbf  T}_u(F_{n+1}), \hat{ \mathbf T}_u(F_{ij})\right)$ for $j=1, \ldots n_i, \ i=0,1$\;
Fixed a number $r$ of neighbours.  By majority vote, the label $Y_{n+1,u}= 0 \textrm{ or } 1$ is assigned to the new observation by the  $k$-nearest neighbour  rule with $k=r$. This procedure is performed for each direction $u$\;
Output $Y_{n+1}=0$  if  $ \frac{1}{k+1} \sum_{s=1}^{k+1}Y_{n+1,u_s} < 1/2$, and $Y_{n+1}=1$ otherwise.
\caption{Classification algorithm for discrete tomography.}
\end{algorithm}            
\end{center}

\subsection{On the number of projections required in practice} 
If the support of the distributions is not known, but only its cardinality $k$, then both Heppes' theorem and our Theorem~\ref{teo1} require the  use of  $(k+1)$ projections. Although this value does not depend on the data dimension $d$ directly,
in practice, for 
high-dimensional binary data, the number of projections may be very large and therefore computationally hard.  According to our experience, if one takes  a sequence of increasing smaller values, the accuracy of the performance on the testing sample increases up to some point and then stabilizes. This suggest using a penalized procedure to select the number of random projections, providing a trade-off between the number of projections and the accuracy of the procedure.

%%%%%%%%%%%%%%%%%%%%%%%%%%%%%%%

\section{Multivariate binary data}\label{binary}

\subsection{Notation}

Let $\X = (X_1, \ldots, X_d)$ be a random vector such that each coordinate $X_i$, $i=1, \ldots, d$ is a Bernoulli random variable, without any independence assumption about the coordinates $X_i$. The distribution of ${\mathbf{X}}$ is called a multivariate Bernoulli distribution and has been studied by several authors, (see for instance \cite{teugels1990}, \cite{fontana2018}, \cite{marchetti2016}, \cite{euan2020}   \cite{dai2013}, and the references therein), where  nice characterizations of the distribution are given.  Also, the simulation problem for this type of distribution for different correlation structures is addressed in the literature, see for instance \cite{oman2009}, \cite{huber2019}, \cite{jiang2020} and \cite{fontana2021}. 

The random vector $\mathbf X$ takes values in the space
\begin{equation}
\Omega_d := \{0,1\}^d = 
\{ \omega =(y_1, \ldots, y_d), \ y_i \in \{0,1\},  \ i=1, \ldots, d                   \},
\end{equation} 
i.e. the set of $d$-tuples of zeros and ones, where $y_i$ stands for the result of the $i$--th trial.

Let $\Delta_d$ be the set of all possible probabilities on $\Omega_d$, i.e.
$$
\Delta_d := \Bigl\{ (p_1, \ldots, p_{2^d}): p_i \geq 0, \ \ \sum_{i=1}^{2^d} p_i = 1\Bigr\}.
$$

For each probability $p \in \Delta_d$, given by 
$$
\{p(\omega) : \omega \in \Omega_d\},
$$
we can derive the law of each $X_i$, as well as the law of $S_d:= X_1 + \ldots + X_d$, which are given by
\begin{align*}
q_i:= P(X_i=1) &= 
 \sum_{\{\omega \in \Omega_d: y_i=1\}} p(\omega),\\
 P (S_d = k) &= \sum_{\{\omega \in \Omega_d: \sum_{i=1}^d y_i = k\}} p(\omega),
\end{align*}
respectively.

Some nice and more explicit formulae are given in \cite{dai2013}, which can be used to perform a statistical analysis of the data in the case 
where we have a sample of iid vectors $\mathbf X_1, \ldots, \mathbf X_N$ with a multivariate Bernoulli distribution.

\subsection{Testing when we have an iid sample of high dimensional binary data}
\label{test4}

Suppose now that we are given a random  sample $\{\mathbf{X_1},\ldots, \mathbf{X_N}\}$,  
where ${\mathbf{X_i}}= (X_{i1}, \ldots, X_{id})\}$, for $i=1, \ldots, N$, and we want to perform some testing problems regarding this sample. This problem becomes difficult for high-dimensional data.
However, several important problems fall into this category. This is the case, for instance, if we have a set of individuals $N$ and for each of them $X_{ij}$ represents the result of a given medical test (positive or negative), image analysis, or in marketing studies, among many others.

In this setting, since the support of the distributions is known in advance, we can make use of Proposition~\ref{unsoloH},
which reduces the problem to checking a single projection on a appropriate subspace $H$. This will be particularly important in the case where we only have access to the projections of the data, as considered in the next sections.

First we choose a one-dimensional subspace fulfilling the condition in Proposition~\ref{unsoloH}.

\begin{description}
	\item [One-sample problem] Given an arbitrary model $P_0$, and and a sample $\{\mathbf{X_1},\ldots \mathbf{X_N}\}$,  where ${\mathbf{X_i}}= (X_{i1}, \ldots, X_{id})$, for $i=1, \ldots, N$, let $P_{0H}$ be the projection of $P_0$ on $H$, and let $P_{NH}$ be the empirical distribution of the projections of the sample on $H$.
	
	Then, the problem of testing $(P=P_0)$ vs $(P \neq P_0)$, reduces to testing  $(P_{H}= P_{H0})$ vs  $(P_{H}\neq P_{H0})$, which can be handled using the Kolmogorov--Smirnov test or the Cramer--von-Mises test, among other possibilities.
	
	For instance, for the Kolmogorov--Smirnov test, let $KS(P_{NH},P_{H0}):= \sup_x \|F_{NH}(x)- F_{H0}(x)\|$, the Kolmogorov distance between the one-dimensional distribution functions of $P_{NH}$ and $P_{H0}$. Reject the null assumption at level $\alpha$ if  $KS(P_{NH},P_{H0}) > c_{\alpha}$. 
	
	The Kolmogorov--Smirnov test for discrete data has been analyzed by several authors. See for instance \cite{conover1972} and \cite{dimitrova2020}. In the latter article, a software package is available to calculate the  exact critical value $c_{\alpha}$ for discrete distributions.

	\item [Two-sample problem] Given two samples $\{\mathbf{X_1},\ldots \mathbf{X_N}\}$, $\{\mathbf{Y_1},\ldots \mathbf{Y_M}\}$ with distributions $P$ and $Q$ respectively, we want to test $(P=Q)$ vs $(P \neq Q)$. Let $P_{NH}$ be the empirical distribution of the projections of the first  sample on $H$, and let $P_{MH}$ be the empirical distribution of the projections of the second sample on $H$.
	
	Again the problem reduces to testing if $(P_H = Q_H)$ vs $(P_H \neq Q_H)$, which can be handled using the two-sample Kolmogorov--Smirnov test or the Cramer--von-Mises test, among other possibilities.

	In this case, we consider the Kolmogorov distance between the two empirical distribution functions, i.e., $KS(P_{NH},P_{H0}) := \sup_x \| F_{NH}(x)- F_{MH}(x)\|$, and reject the null assumption at level $\alpha$ if $KS(P_{NH},P_{MH}) > c_{\alpha}$.
	\end{description}
	
\noindent  Some practical examples, like testing for the Poisson--Binomial distribution, are considered in Section \ref{simus} below.

\section{Bernoulli series distributions and their sums}\label{unasola}

\subsection{Background}\

The problem we want  to analyze  in this section is 
whether we can still say  something when we have only one realization $\mathbf X$ of a multivariate Bernoulli distribution. 
 In particular, we are interested in characterizing the law of the sum $S_d$ for different distributions $p \in \Delta_d$,
  in particular those for which the law of $S_d$ is close to a Binomial distribution $B(d, 1/2)$ with parameters $d$ and $1/2$.

Let $L_d(p)(k) = P(S_d = k)$.
Given $\epsilon > 0$ let
\begin{equation}\label{bigset}
\Delta(d;\epsilon) := \Bigl\{ p \in \Delta_d: \max_{0 \leq k \leq d} \vert L_d(p)(k) - B(d, 1/2, k) \vert \leq \epsilon\Bigr\},
\end{equation}
the set of probabilities $p \in \Delta_d$ which are uniformly close to  $B(d, 1/2)$.

Next we  choose at random a probability measure $p \in \Delta_d$ according to the normalized Lebesgue measure $\mu$ on the simplex $\Delta_d$, which is related to the notion of physical entropy.

In this setting, we have the following result of \cite{chevallier2011}.

\begin{theorem}[\protect{\cite[Theorem 1]{chevallier2011}}]
There exists a constant $A\le 2$ such that
 \begin{equation}\label{bound} 
 \mu(\Delta(d; \epsilon)) \geq 1 - \frac{A \sqrt{d}}{\epsilon^2 2^{d-1}}
 \end{equation}
 and
 \begin{equation}\label{bound2}
 \mu\Bigl( \{ p \in \Delta_d:  \sup_{I \subset \{1, \ldots, d\}} \vert L_d(p)(I) - B(d, 1/2, I) \vert \leq \epsilon\}\Bigr) \geq 1 - \frac{A d^{5/2}}{\epsilon^2 2^{d-1}}.
\end{equation}
\end{theorem}

\begin{remarks}

	\
	
	\begin{itemize}
		\item [1)] Observe that, since we are choosing at random a probability measure $p \in \Delta_d$ according to the normalized Lebesgue measure on the simplex $\Delta_d$, the mean value $\mathbb E(S_d)= d/2$.  
		\item [2)] Even for moderate $d$, the bound is very close to $1$, and the  results are uniform. For instance, if $d=30$ and $\epsilon=0.01$, the right hand side of (\ref{bound}) is greater than 0.99979, and for (\ref{bound2}), if $d=50$ and $\epsilon=0.01$, the bound is greater than 0.9999.
		For most probability distributions $p \in \Delta(d, \epsilon)$, the distribution of $S_d$ will be much the same, and so little information can be derived from its distribution. In other words, with high probability the distribution of $S_d$ will be very close to $B(d,1/2)$ if we choose at random the joint distribution $p$ of the random variables $X_1, \ldots, X_d$. This means that  the class of probabilities $p$ which are not close to the $B(d, 1/2)$ is very small.   
		\item [3)] However the set of ``rare" distributions $p$ contains some well-known structured distributions. For instance if the random variables $X_i$ are iid Bernoulli with common parameter $q \neq 1/2$, the law of $S_d$ is not close to the $B(d, 1/2)$. Moreover, if the random variables are independent but with different parameters,
		$$
		p(w) = \prod_{\{y_i=1, 1\leq i \leq d\}} q_i \prod_{\{y_i=0, 1\leq i \leq d\}} (1- q_i),  \qquad w \in \Omega_d,
		$$
		then the law of $S_d$ is a Poisson--Binomial distribution, given by
		
		$$
		P(S_d=k)=\sum  _{A\in \Lambda_k} \prod  _{i \in A}q_i \prod _{j\in A^c}(1-q_{j}),
		$$
		where $\Lambda_k$ is the family of all subsets of $\{0, 1, \ldots, k\}$.
		This typically is also far from  $B(d, 1/2)$. For instance, this will be the case if $ \bar q:=E(S_d)/d = \frac{1}{d} \sum_{i=1}^d q_i \neq 1/2$, where $q_i = P(X_i=1)$. %Moreover, even if $\bar q = 1/2$, from Theorem 1 in Ehm (xx) we have that the total variation distance $d_{TV}(.,.)$ between the Binomial($d, 1/2$) and the Poisson Binomial distribution is larger than
	%	$$
	%	\frac{\sum_{j=1}^d (q_j - 1/2)^2}{ 31 d},
%		$$
%		which will not be small  except for the case where all of the $q_j$ are very close to $1/2$ except for a vanishing fraction $\alpha_d$ of them.  
	 
	\end{itemize}
	
	\end{remarks}
	
What if we assume that $\mathbb E(S_d)=d/2$?

	Even if $\mathbb E(S_d)=d/2$, from Theorem 1 in \cite{ehm1991}  we have that the total variation distance $d_{TV}(\cdot,\cdot)$ between the Binomial($d, 1/2$) and the Poisson Binomial distribution is larger than
	$$
	\frac{\sum_{j=1}^d (q_j - 1/2)^2}{ 31 d},
	$$
	which will not be small  except for the case where all but a vanishing fraction $\alpha_d$ of the $q_j$ are very close to $1/2$.

	From a statistical point of view we have to be very careful when dealing with high-dimensional binary data, if the data are not independent. Indeed, in contrast to the case where we have a Binomial$(d,q)$ distribution, where we have a sample of size $d$ of iid Bernoulli$(q)$ random variables, we have only one sample of the distribution $p$ in a high-dimensional binary space. However, these results provide an excellent way to perform some statistical tests for important applications. Indeed, we can use these results for testing if the distribution of $S_d$ has some structure, which in terms of our setting will correspond to be in the set of ``rare" distributions.

\subsection{Testing when we have only one high-dimensional binary datum}
\label{serie5}
\

Given a sample $X_1, \ldots, X_d$ of Bernoulli random variables, let $S_d = \sum_{i=1}^d X_i$ and let $0<\alpha<1$.
We start by fixing $\epsilon$ such that 
\begin{equation}
\label{rel}
1-\frac{2 \sqrt{d}}{\epsilon^2 2^{d-1}} \geq 1 - \alpha/2.
\end{equation}
For $\alpha=0.01$ and $\epsilon=0.01$ this will hold for any $d \geq 26$, for $\epsilon=0.005$ for any $d \geq 28$ and for $\epsilon=0.001$ for any $d \geq 36$. 

We consider as the null assumption that the distribution $p$ is not rare, i.e. that it belongs to the set $\Delta(d,\epsilon)$, a set that satisfies
\begin{equation}\label{nullhip}
\mu(\Delta(d; \epsilon))  \geq 1 - \alpha/2,
\end{equation}
and next we look if the observed value for $S_d$ is very atypical for a $B(d, 1/2)$.

Because of the definition of $\Delta(d, \epsilon)$, a conservative test of level $\alpha$ can be performed, rejecting the null assumption if
\begin{equation}\label{test}
S_d \notin [l,r],
\end{equation}
where $P(B(d,1/2) \notin [l,r]) \leq \alpha/2$.

Rejecting from the right ray or from the left ray will have different implications in some applications.

In case that we reject the null hypothesis, we would like to characterize in some way the rare distribution, but it will not be possible with just a single sample of the vector of binary data. In order to do so, we need to have an iid (or at least mixing) sample of the binary vectors. %This case is consider in the previous sections.
%\section{}

\subsection{Testing for rare distributions} \

We proceed as follows. For each vector $\mathbf{X}_i$ let $S_{id} = \sum_{j=1}^d X_{ij}$ the sum of the coordinates of the vector $\mathbf X_i$, $i=1, \ldots N$.
For each $k=0, \ldots, d$, we consider the empirical probability given by
\begin{equation} \label{empirical}
P_N(k):= \frac{1}{N} \sum_{l=1}^N \mathcal I_{\{S_{ld}=k\}}.
\end{equation}                                  
From the Hoeffding inequality (\cite{bennett1962,hoeffding1963}), we can derive that, for any $t>0$,
\begin{equation} \label{hoe}
P\left(\vert P_N(k)- P(S_{1d}=k)\vert > t \right) \leq 2 \exp( -2Nt^2),
\end{equation}
and so, from the union bound,
  \begin{equation} \label{hoe2}
P\Bigl(\max_{0 \leq k \leq d} \vert P_N(k)- P(S_{1d}=k)\vert > t \Bigr) \leq 2(d+1)\exp(-2Nt^2).
\end{equation}                 
Now we are ready to derive a conservative level-$\alpha$ test for the null hypothesis
\[
\text{ H0:   the distribution $p$ is not rare.}
\]
This can be obtained from inequalities (\ref{bound}) and (\ref{hoe}) or (\ref{hoe2}), rejecting the null assumption when:
\begin{equation} \label{eltest}
 \vert P_N(k) - B(d; 1/2,k) \vert   > a,
\end{equation}
where $a$ is chosen to satisfy the inequalities (\ref{ineq1}) and (\ref{ineq2}) below, and where $k$ is the observed value.
Indeed, we have that
\begin{align*}
&P\Bigl( \Bigl|P_N(k) - B(d; 1/2,k) \Bigr|  > a\Bigr) \\
&\leq P\Bigl( \Bigl| P_N(k) - P(S_{1d} = k) \Bigr|    > a/2\Bigr)
 + P\Bigl( \Bigl| P(S_{1d}=k)- B(d; 1/2,k) \Bigr|   > a/2\Bigr).
 \end{align*}
The first term on the right hand side will be smaller than $\alpha/2$ provided that $N$ and $a$ are
chosen so  that 
\begin{equation} \label{ineq1}
\exp(-2 N a^2)< \alpha/2,
\end{equation}  
while the second one will be smaller than $\alpha/2$ if $d$ and $a$ are chosen so that
\begin{equation} \label{ineq2}
\frac{8 \sqrt{d}}{a^2 2^{d-1}} < \alpha/2.
\end{equation}
An easy calculation leads to
\begin{equation}\label{valorcrit}
a= \max\left\{ \sqrt{\frac{-\log(\alpha/2)}{2N}},~\frac{d^{1/4}}{2^{(d-5)/2} \sqrt{\alpha}}\right\}.
\end{equation}
For instance, if $d=20$, $N=200$ and $\alpha=0.05$, then $a=\max(0.052, 0.086)$. 
For $d\geq 30$ the second term is negligible ($<0.0018$), and the bound is driven by the first term.

\section{Simulations} \label{simus}
\subsection{Example 1: Binary data classification}
\label{Ex:binary}
We generate two different multivariate Bernoulli distributions in dimensions 
$d \in \{5,10,15,20\}$.
Both distributions have marginals Bernoulli$(1/2)$. In one of them, the components are independent, while, for the other one, we consider correlated components with a parameter $\operatorname{Cor}$ taking values in the set $\{0.1, 0.3,0.5, 0.7,0.9 \}$ for each different scenario. 

We generate 200 observations for each class and $25\%$ of them are used for the testing sample. The data are generated as indicated in \cite{park1996}. The distances between the empirical distributions  $P_{n_i,i,u_l}$ and  $P_{{n_i}+1,i,u}$ are calculated according to the statistic   DTS, given for instance in \cite{dowd2020}.

Next we apply Algorithm \ref{alg1} to perform the classification. In all cases we use 100 one-dimensional projections. 
For each of the considered scenarios, 100 replicates are performed. The misclassification error of Random Forest,  \cite{needell2018} (denoted as N-S-W method) considering $L=3$ ``levels'', and our proposal are given in Table~\ref{tableber}, while the boxplots of the misclassification errors of our proposal are reported in Figure~\ref{fig21} for different numbers of projections. One can see that, when the correlations of the second group become higher, the distributions are ``further apart", and the misclassification error becomes smaller. Also, we observe that, when the dimension of the space is higher, the algorithm performs better. 

\begin{table}
\tiny
\caption{Average misclassification errors for each scenario in the Example \ref{Ex:binary}  over $1000$ replicates. Standard errors are reported in brackets.\label{tableber}}
\centering
\begin{tabular}{|c|cccc|cccc|cccc|cccc|}
\hline
  \textbf{mean}   & \multicolumn{4}{|c|}{RF}   & \multicolumn{4}{|c|}{N-S-W } & \multicolumn{4}{|c|}{$\textrm{RP}_{10000}$} \\   \cline{2-13} 
   \textbf{(sd)} & \multicolumn{12}{|c|}{Dimension} \\   \cline{1-13} 
  Corr. & 5 & 10 & 15 & 20& 5 & 10 & 15 & 20& 5 & 10 & 15 & 20 \\\cline{1-13} 
 \multirow{2}{*}{  0.1 }&46.7 &46.9 &46.5 & 46.1& 46.2 &47.2 &47.7 &46.8& 46.6 & 47.4 & 46.5 &45.8 \\
 &(5.0) & (5.1) & (5.2) & (5.2)&(5.1) & (5.5) & (5.0) & (5.0) &(4.6) & (5.0) & (5.1) & (4.6) \\\cline{1-13}
   \multirow{2}{*}{   0.3} & 34.8 & 32.7 & 32.0 & 31.7& 36.0 &33.3& 32.6 &31.5& 35.4 &34.0 &31.7& 31.8 \\
   &(4.9) & (4.4) & (4.3) & (4.2)&(5.3)  &(4.2)  &(4.4) & (4.1)& (4.8)  &(4.5)  &(5.3)  &(4.5)\\\cline{1-13}
      \multirow{2}{*}{  0.5}& 24.5 &19.6& 18.8 &17.9 & 24.7 &20.7 &19.2 & 18.3 & 24.8 &20.2& 19.3 &17.6 \\
      &(4.7) & (3.9) & (3.7) & (3.2)&(4.7) & (3.9) & (4.2) & (4.6)&(5.1) &  (4.0) &  (3.7)&  (3.7)\\\cline{1-13}
       \multirow{2}{*}{   0.7} & 17.1 &10.1 &7.4 &6.9 & 17.5 &10.8& 9.1 &8.5& 17.2 &10.7 &9.1 &8.5 \\
       &(3.7)  &(2.9) &(2.7)& (1.9)& (3.7) & (3.2)& (2.9) &(3.0)& (3.9) & (2.8)& (3.2) &(2.8)\\\cline{1-13}
          \multirow{2}{*}{  0.9} & 9.7 &3.9 &3.3 &3.0& 9.7 &4.5 &2.8& 2.5& 9.7 &3.9& 2.8 & 2.2 \\
          &(3.0) &(2.1) &(1.5) &(1.2)&(3.1) &(2.5)& (1.8) & (1.6)& (3.0) & (2.0) & (1.8) & (1.6) \\ \cline{1-13} 
\end{tabular}
\end{table}

\begin{figure}[htb]  
\centering
\includegraphics[width=\textwidth]{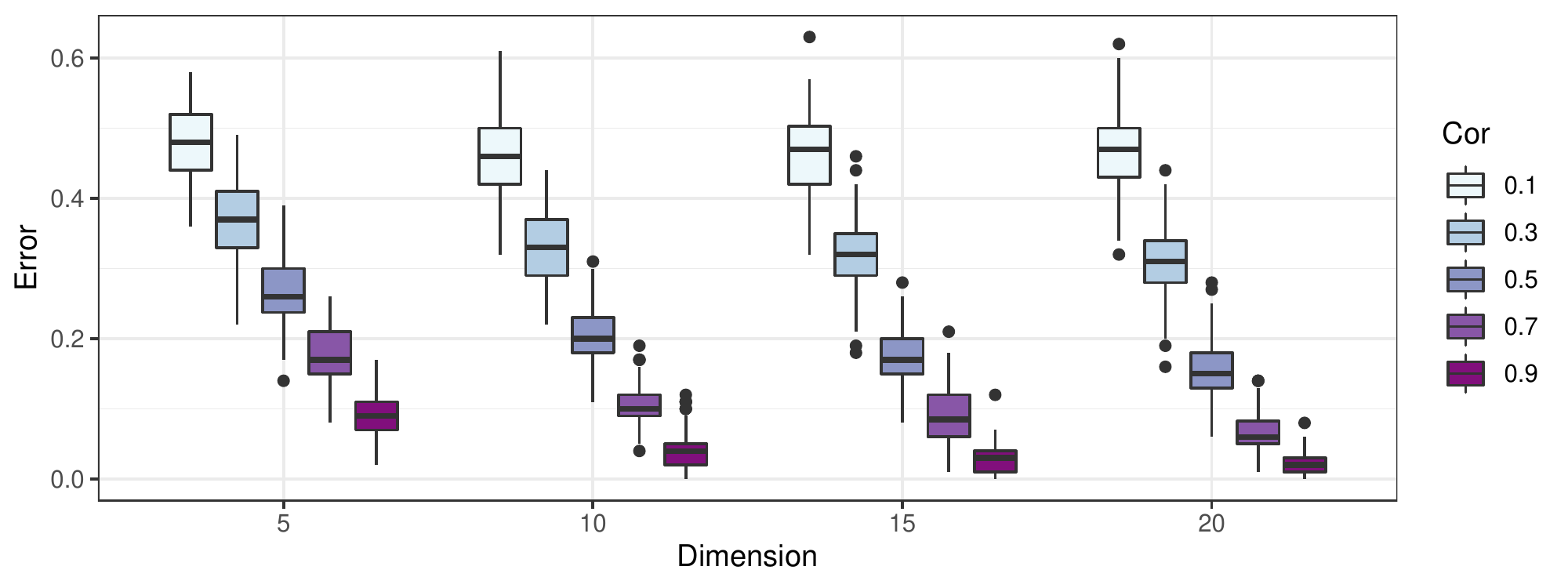}
\caption{Boxplots of the misclassification errors for each scenario based on 100 replications.  } \label{fig21}  
\end{figure}

\subsection{Example 2: Classification of discrete Tomography by using distances between histograms }

In this example we generate a set of ``phantom images'' using two different random patterns. 
In the first scenario, the first sample (of size $200$) each image is built with $5$ circles of centers  $(2,2)$, $(-2,2)$,$(2,-2)$,$(-2,-2)$ and $(0,0)$ with a normal random radius with mean $1$ and standard deviation $1/10$.
The second sample (of size 200) is built up in the same way, but adding a circle centred at $(0,2)$ and with a  normal random radius with mean $1/2$ and standard deviation $1/10$ (see Figure \ref{figg3}).  The points that define the image are in a equi-spaced grid at distance $0.05$.

\begin{figure}[htb]  
\centering
\includegraphics[width=55mm]{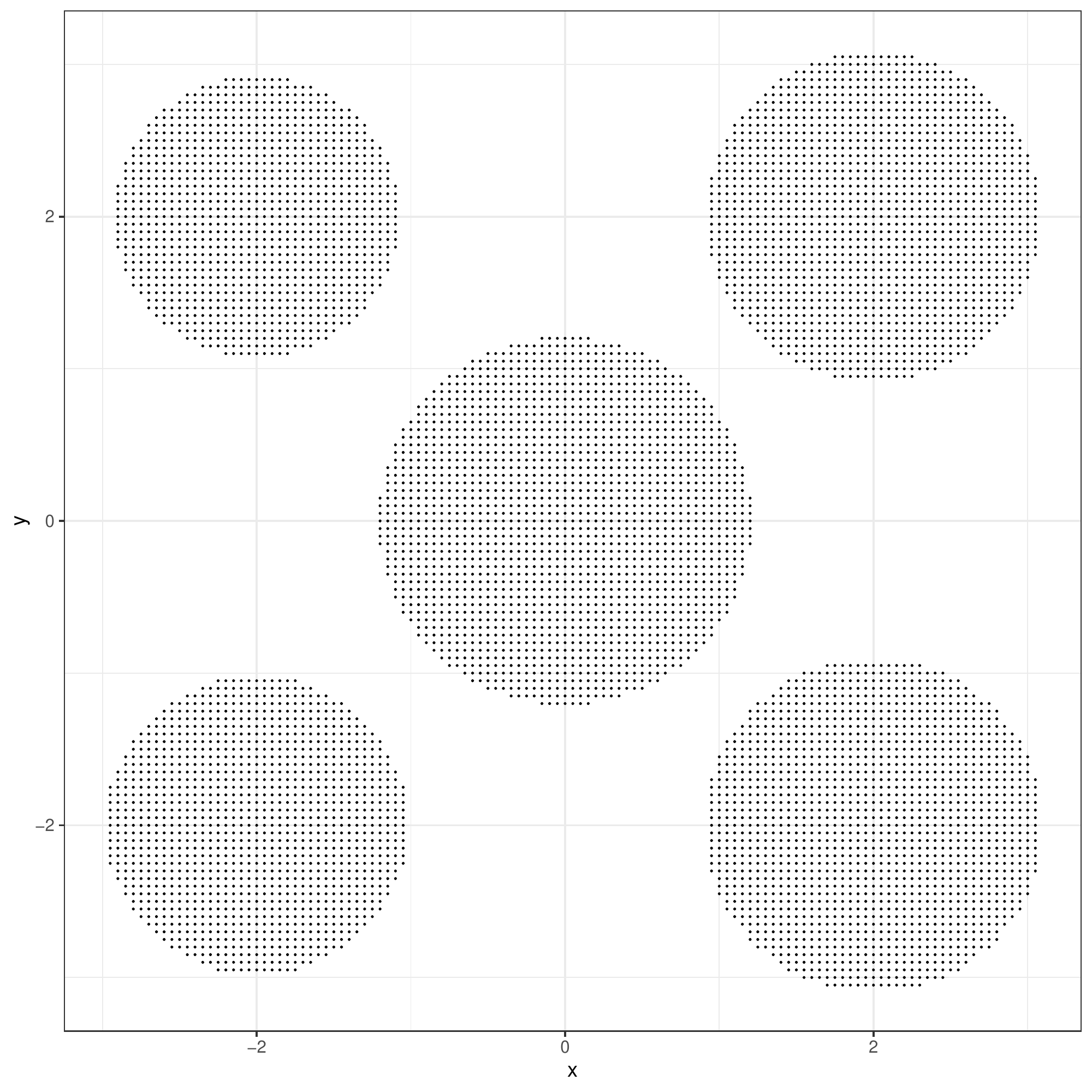}
\includegraphics[width=55mm]{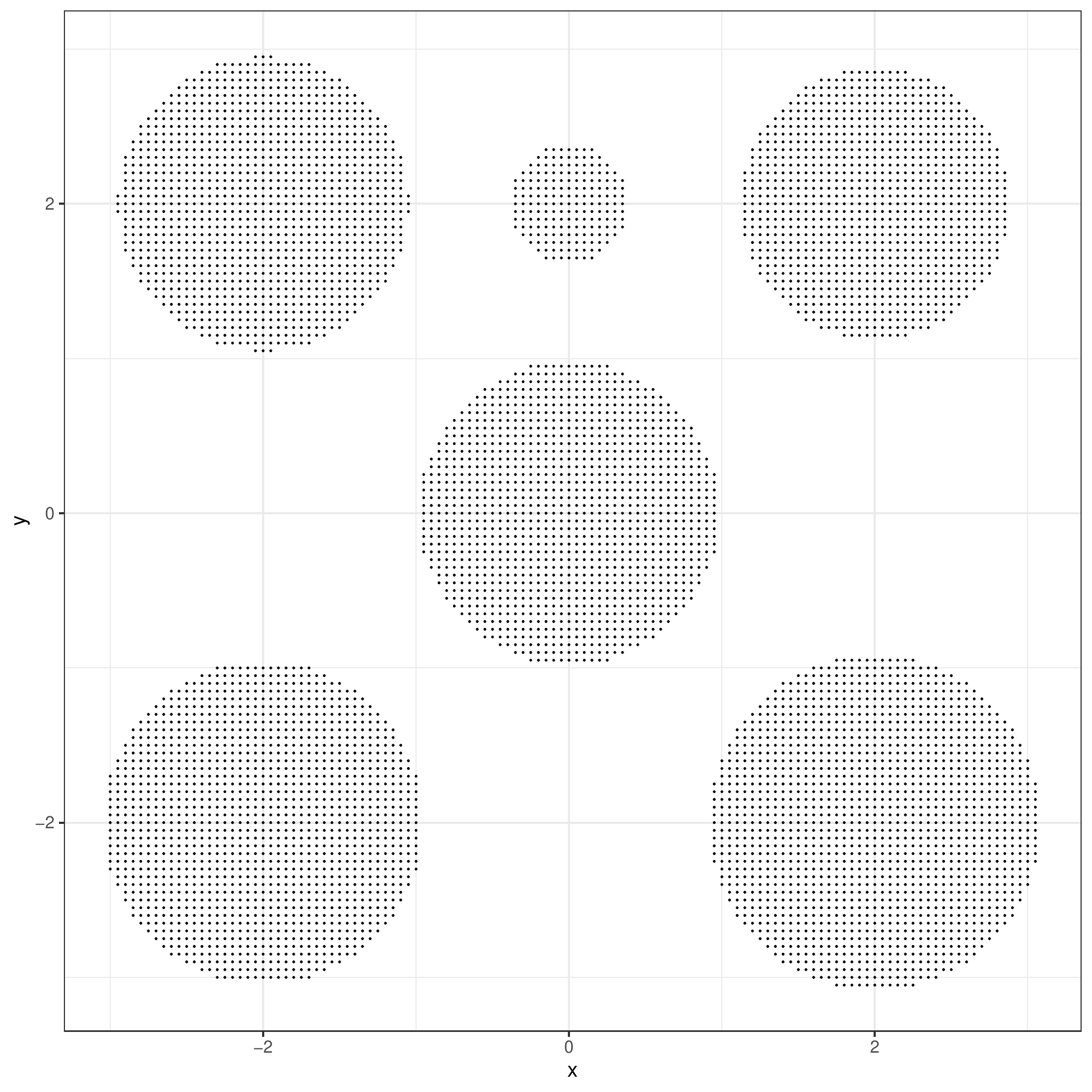}
\caption{Representation of an image of each group in scenario $1$.} 
\label{figg3} 
\end{figure}

For the second scenario, the first sample is equal to that of the first scenario.
Now the second sample consists on $5$ circles with centres $(2,2)$,  $(-2,2)$,$(2,-2)$,$(-2,-2)$, $(0,0)$ and random radius with mean $1.2$ and standard deviation $1/10$, see Figure \ref{figg4}. 

The value $k$ for the k-NN is the one that minimizes the misclassification error in the testing sample ($k=21$).

\begin{figure}[htb]  
\centering
\includegraphics[width=55mm]{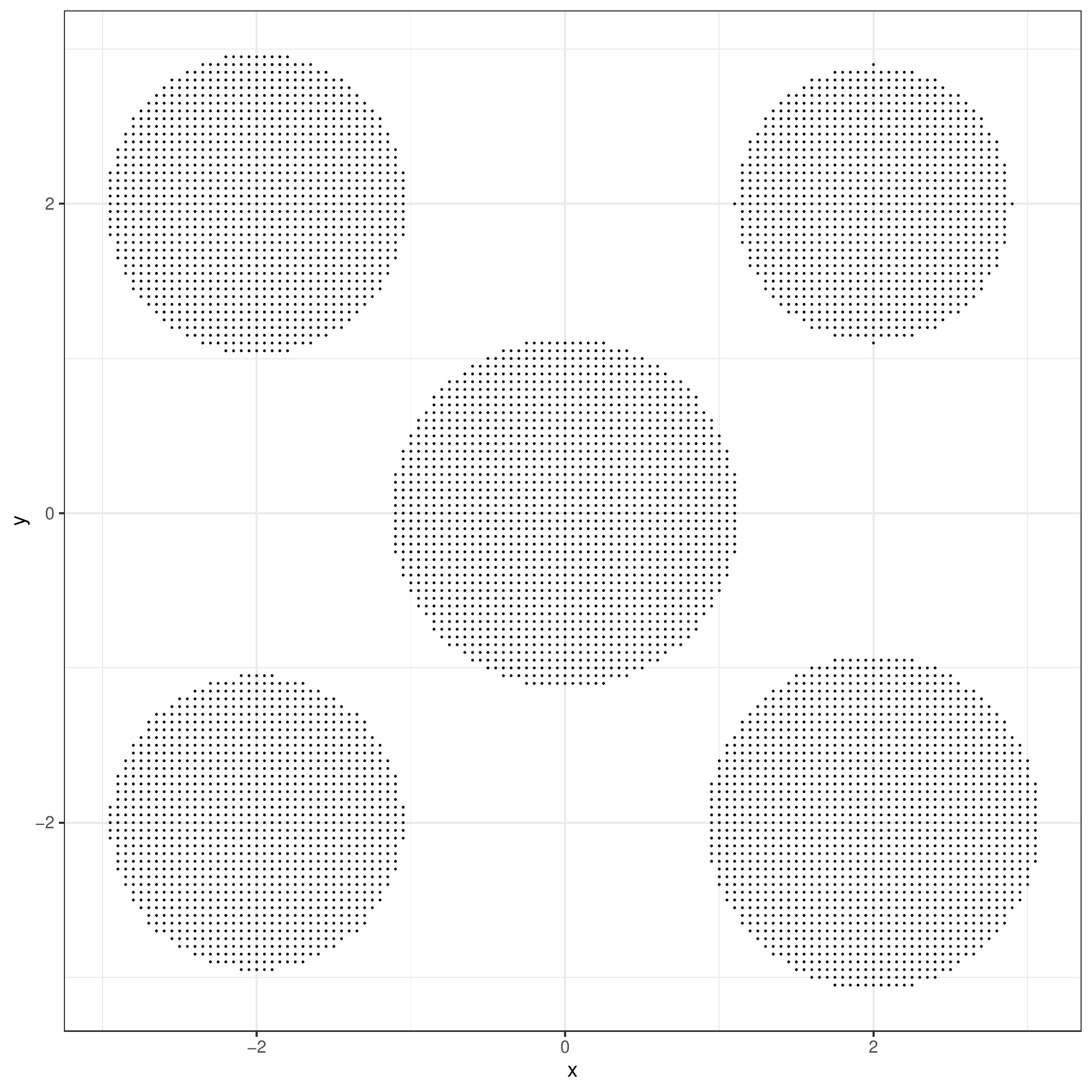}
\includegraphics[width=55mm]{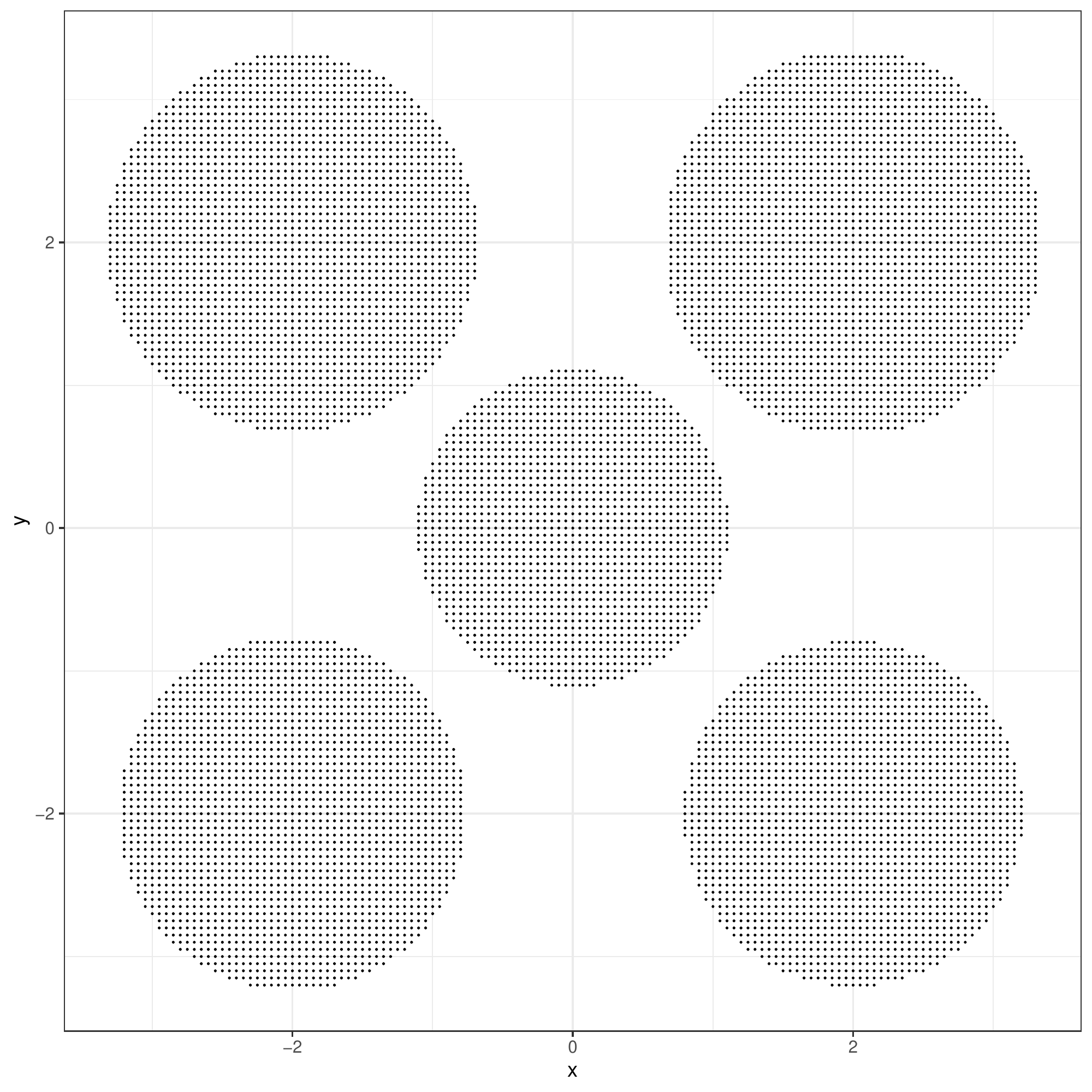}
\caption{Representation of an image of each group in scenario $2$.}
\label{figg4}
\end{figure}

$75 \%$ of the samples are considered for the training sample and the rest for the testing sample. We use Algorithm~\ref{alg4} with $100$ projections at random in each case. 

For each direction we project the points of the image and build up an histogram. Given a testing sample using $k_NN$, we observe the proportion of neighbours at each group. If the average over all directions of the proportions of each group is greater for group~1, then the image is classified as being from group~1, otherwise is classified as from group~2. As mentioned above, the distance between the histograms is calculated as indicated in  \cite{irpino2006}, i.e., using the Mallow distance in $L^2$.  
 
The misclassification rate in the testing sample are $2.55 \%$ and $6 \%$ in scenarios 1 and 2 respectively.

\subsection{Example 3: A test for binary data}

We consider a sample of Bernoulli vectors just  as in section \ref{test4}, $\{\mathbf{X_1},\ldots \mathbf{X_{200}}\}$  where ${\mathbf{X_i}}= (X_{i1}, \ldots, X_{id})\}$, for $i=1, \ldots, 200$, with $d=8$, and all marginals $X_{j} \sim \textrm{Bernoulli}(1/2)$ for all $j$.

The testing problem is the following: under the null assumption, the components of the vector are independent, i.e. 
\[
P_0( X_{1}=j_1, \ldots, X_{d} =j_d)= \Bigl( \frac{1}{2} \Bigr) ^d, \quad j_1, \ldots,j_d\in\{0,1\}.
\]

Under the alternative the data are simulated using the odds ratio
$$OR_{ij}= \frac{P\left( X_{i}=1, X_{j}=1 \right) P\left(X_{i}=0, X_{j}=0  \right)   }{P\left( X_{i}=0, X_{j}=1 \right) P\left(X_{i}=1, X_{j}=0  \right) }= \gamma,$$ 
with $\gamma\in [1,3]$ if $i \neq j$ and $+\infty$ if $i=j$. The larger the value of $\gamma$,
the further  the alternative is from the null. 

The data are projected into $k \in \{ 1,10,50,100,500\}$ random directions $u_1, \ldots, u_{k}$. The testing statistic is given by $KS(\gamma)= \frac{1}{k}\sum_{j=1}^{k} KS(P_{Nu_j},P_{H0})$. The distribution of the statistic under the null is obtained by Monte Carlo. The power function is obtained using $1000$ replicates as a function of $\gamma$, see Figure \ref{ftest1}.

Observe that, with just $50$ projections, for  $\gamma \geq 1.75$ the power of the test is greater than  $73 \%$. The simulations were performed using the R package \textbf{mipfp}, see \cite{barthelemy2018}.

\begin{figure}[htb]  
\centering
\includegraphics[width=\textwidth]{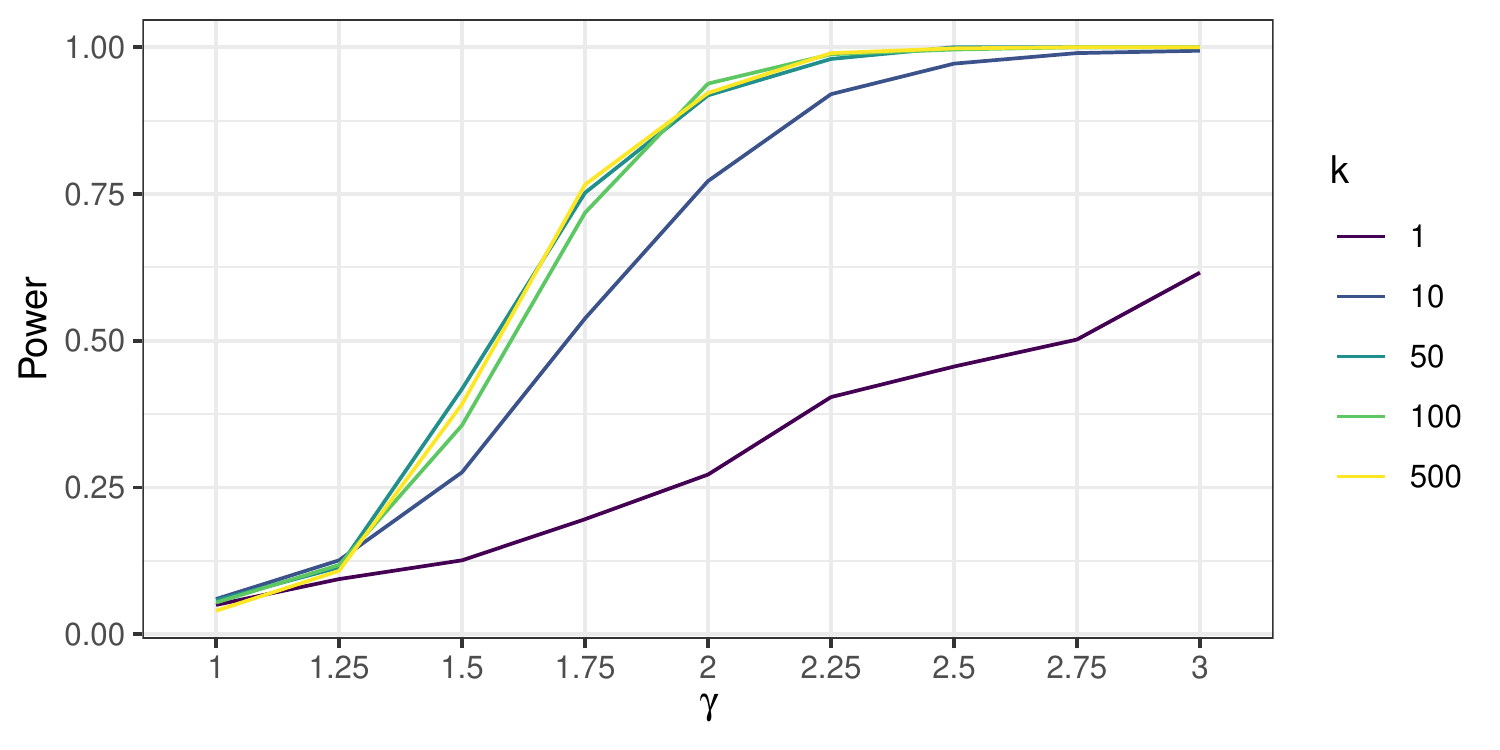}
\caption{Estimated power function for $\gamma\in [1,3]$ for different number of projections.} \label{ftest1}  
\end{figure}

\subsection{Example 4: Test for Poisson-Binomial}

Next we consider the case of Poisson-Binomial $S_d$ distributions with $d \in \{50,100,200,500,1000\}$, i.e. the sum of $d$ independent Bernoulli random variables with different parameters.  The $d$ probability parameters of the distribution are generated from a Beta distribution with parameters  $\gamma_1 \in [2,4]$ and  $\gamma_2 =2$. If $\gamma_1 =\gamma_2$, then
$$ 
\frac{1}{d} E \left( S_d \right) = \frac{1}{d} \sum_{i=1}^{d} q_i \approx  \frac{\gamma_1}{\gamma_1+\gamma_2} =\frac{1}{2},
$$ 
for $d$ large enough.

As indicated in Section \ref{serie5}, the power of the test is estimated via Monte Carlo for each value of $d$ as a function of $\gamma_1$, see Figure \ref{ftest2}. The behaviour becomes better when increasing the dimension $d$, while the problem becomes easier by increasing $\gamma_1$.

\begin{figure}[htb]  
\centering
\includegraphics[width=\textwidth]{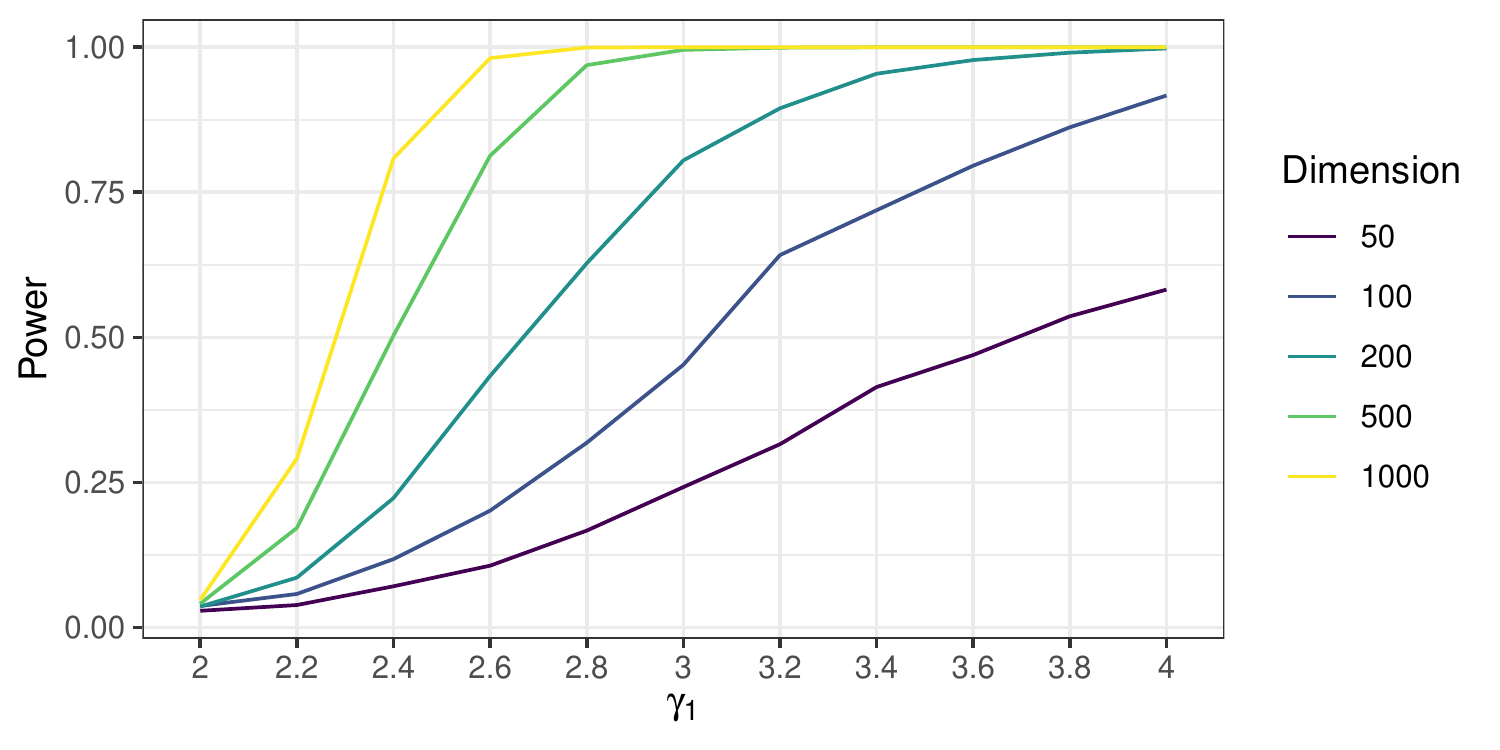}
\caption{Power function for $\gamma_1 \in [2,4]$. } \label{ftest2}  
\end{figure}

\section{Real-data examples}\label{S:realdata}

In this section three examples with real data are exhibited. In both cases the data are a set of binary vectors and Algorithm $1$ is performed.

\subsection{Example 5}

The data describe the diagnostic of  Cardiac Single Proton Emission Computed Tomography (SPECT) images. Each patient is classified into one of two categories: normal and not normal. The database has 267 images called SPECT corresponding to each patient, ``that measures radioactive counts that represent perfusion in the LV muscle in a specified ROE" (regions of interest) resulting 22 regions of interest, see \cite{kurgan2001}.

The images are processed and transformed into a vector of dimension $44$ of continuous variables. Later on, this pattern is processed again to obtain a binary vector of dimension $22$ for each patient. 

The sample is split into $80$ observations for the training sample and $187$ for the testing sample. The database is available from the UCI repository\footnote{\url{https://archive.ics.uci.edu/ml/datasets/spect+heart}}.

We apply Algorithm $1$ to classify the data using different numbers of random projections ($k \in \{5,10,20,50,100,500,1000,10000 \}$). 

The procedure is replicated $100$ times at each scenario and the boxplots of the misclassification rates  are given in Figure \ref{corazon}.

\begin{figure}[htb]  
\centering
\includegraphics[width=\textwidth]{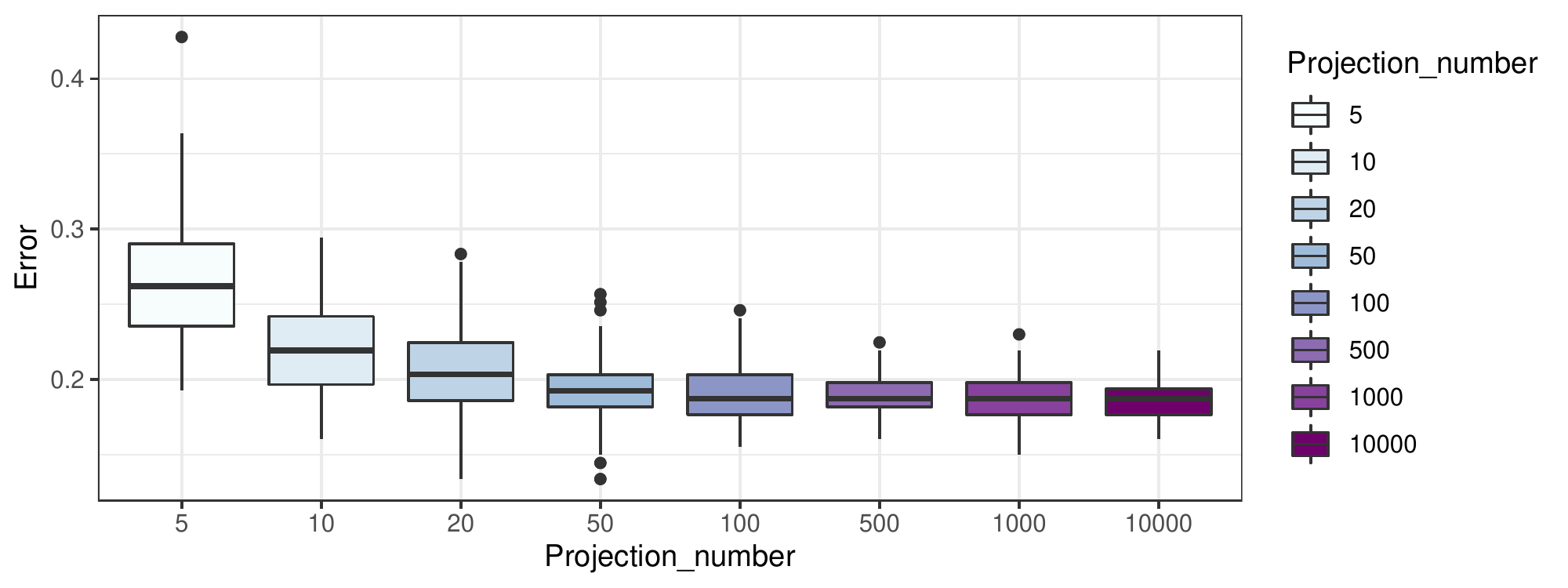}
\caption{Boxplot of the misclassification rates obtained from the testing sample for different numbers of random projections. } \label{corazon}  
\end{figure}

The performance of the method that we propose (denoted by $\textrm{RP}_k$) is compared with two classical classification methods (Random Forest and SVM) and in particular with the proposed methodology for binary data in N-S-W (with $L=3$). Table~\ref{tablecorazon} shows the confusion matrices obtained in the test sample for each method. The misclassification rates obtained from the testing sample are $23 \%$, $24.1 \%$, $19.3 \%$ and $18.4 \%$  for Random Forest, SVM ,N-S-W, and $\textrm{RP}_k$ respectively. Therefore, in this example, methods designed for discrete data perform better than some general classification methods. In addition, the N-S-W and $\textrm{RP}_k$ methods have similar behaviour.

\begin{table}
\caption{Confusion matrices (in percentages) in the test samples of binary classifications. Here $0$ is normal and $1$ and is not normal.\label{tablecorazon}}
\centering
\begin{tabular}{|c|ccc|ccc|ccc|ccc|}
\hline
  & \multicolumn{12}{|c|}{ Labels}  \\ \cline{2-13} 
  & \multicolumn{3}{|c|}{ Random Forest} &\multicolumn{3}{|c|}{ SVM }& \multicolumn{3}{|c|}{N-S-W} & \multicolumn{3}{|c|}{$\textrm{RP}_{10000}$}    \\ \cline{2-13} 
 &  & \textbf{0} & \textbf{1} & & \textbf{0}  & \textbf{1} & & \textbf{0} &  \textbf{1} & & \textbf{0} &  \textbf{1} \\  \hline
\multirow{2}{*}{Pred.} & \textbf{0} & 5.9 & 2.1& \textbf{0}& 6.4 &   1.6& \textbf{0}& 5.3 & 2.7 & \textbf{0} & 5.1  & 2.9 \\
& \textbf{1} & 20.9  & 71.1 & \textbf{1}& 22.5 & 69.5 & \textbf{1} & 16.6 & 75.4 &  \textbf{1} & 15.5 & 76.5   \\ \cline{1-13} 
\end{tabular}
\end{table}

\subsection{Example 6}
In \cite{grisoni2019},  consensus machine learning algorithms are used to predict the binding to the androgen receptor AR. 
		We quote from \cite{grisoni2019}:
		
	``The nuclear androgen receptor (AR) is one of
	the most relevant biological targets of Endocrine Disrupting
	Chemicals (EDCs), which produce adverse effects by
	interfering with hormonal regulation and endocrine system
	functioning.
	The nuclear androgen receptor (AR), whose gene is located on
	the X chromosome, is expressed in a wide range of tissues and
	plays a fundamental biological role in bone, muscle, prostate,
	adipose tissue, and the reproductive, cardiovascular, immune,
	neural, and hemopoietic systems. AR is one of the target
	receptors of the so-called Endocrine Disrupting Chemicals
	(EDCs), exogenous compounds able to disturb hormonal
	regulation and the endocrine system functioning, thereby
	producing adverse effects in humans and wildlife.  EDCs can
	interact directly with a given nuclear receptor and perturbate or
	modulate downstream gene expression,  but they can also have
	direct effects on genes and epigenetic impact.  Disruption of
	AR-mediated processes can cause irreversible consequences in
	human health. For example, some chemicals, like pesticides
	(e.g., DDT), disrupt male reproductive development and
	function by inhibiting androgen-receptor-mediated events. 
	Recently, machine learning (ML) and computer-aided
	techniques have shown to be useful for modeling nuclear
	receptor modulation of chemicals at different levels, such as drug
	discovery and design and testing prioritization campaigns.''

The dataset, which is available at  the UCI repository\footnote{\url{https://archive.ics.uci.edu/ml/datasets/QSAR+androgen+receptor}}, contains agglutinate (positive) and no negative agglutinate molecules.
$150$ of each type are taken as training sample and $45$ of each type as testing sample. Each molecule is represented by $1024$  binary molecular fingerprints, that is, a 1024 length sequence of 0s and 1s, see \cite{grisoni2019,piir2021}.
%The data are available at the UCI repository \footnote{\url{https://archive.ics.uci.edu/ml/datasets/QSAR+androgen+receptor}}. 

The boxplots of the misclassification errors are given in Figure \ref{gen} for different numbers of random projections.

\begin{figure}[htb]  
\centering
\includegraphics[width=\textwidth]{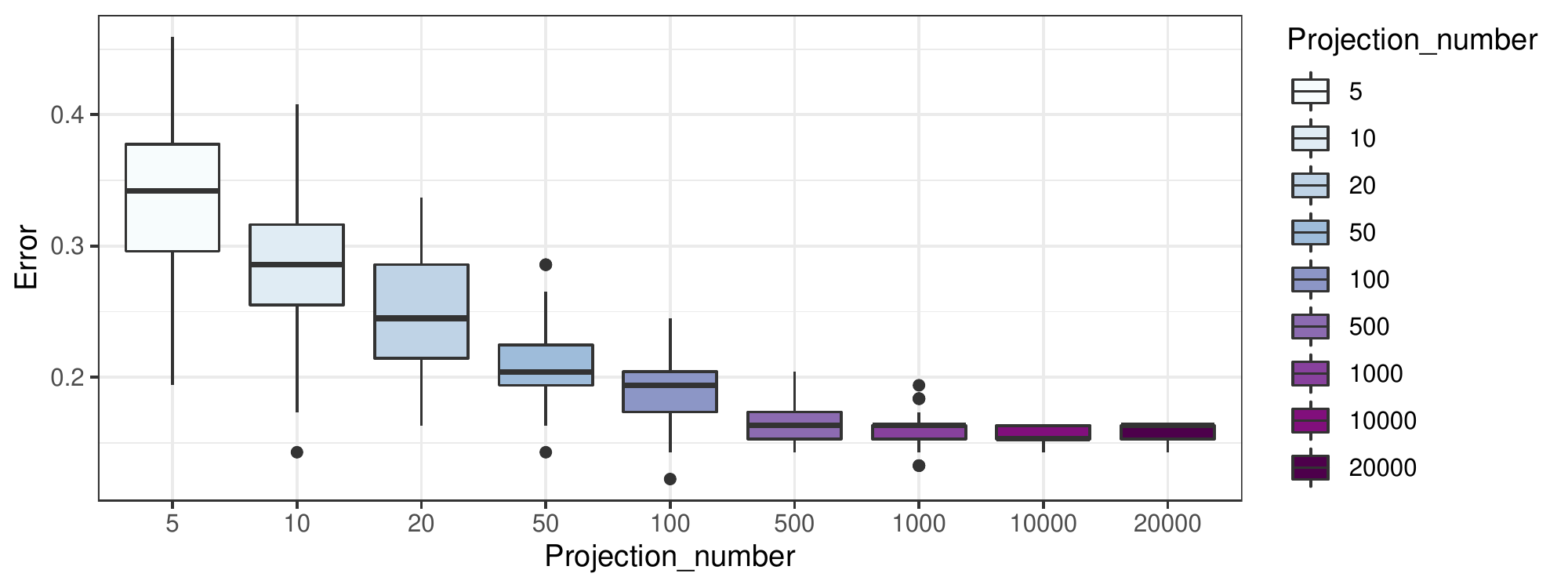}
\caption{Boxplots for the misclassification errors at the testing sample for different number of random projections.} \label{gen}  
\end{figure}

Considering $1000$ projections, the performance of the classification in each class is evaluated through three indices in the test sample:
\begin{description}
\item[Sensitivity] $\textrm{Sn}= \frac{\textrm{TP}}{\textrm{TP}+\textrm{FN}}= 0.97$
\item[Specificity] $\textrm{Sp}= \frac{\textrm{TN}}{\textrm{TN}+\textrm{FP}}= 0.72$
\item[Non-Error Rate] $\textrm{NER}= \frac{\textrm{Sn}+ \textrm{Sp}} {2}= 0.845,$
\end{description}
where TP and TN stand for the misclassification errors in each class. 
These results correspond to the median values over 100 replicates. We observe that, with respect to the results obtained in \cite{grisoni2019}, we get a slightly smaller value for the  specificity, but   better results for the sensitivity and the   non-error rate.

\subsection{Example 7: The MNIST database }
This example is similar to Example $2$, but with real data. The base has a set of handwritten digits and is available in \url{http://yann.lecun.com/exdb/mnist/}. The digits have been size-normalized and centered in a fixed-size image. Each image is represented by a matrix of $28\times 28= 784$ pixels.  Each pixel takes an integer value from $0$ to $255$ (grey scale).  The labels form the components of a vector representing the digit shown in the image.
In this work the image is dichotomized to 1--0 (black and white respectively). If the pixel is above 100 on the grey scale we associate 1 to it, and otherwise 0. As an illustration, between the sample of digits $3$, $4$ and $9$ in a test sample, the classification is performed (as in  the Example $2$). 
In Figure \ref{xx}, three dichotomized images of each of these digits are shown.

\begin{figure}[htb]  
\centering
\includegraphics[width=0.9\textwidth, trim=100 430 100 180, clip=true]{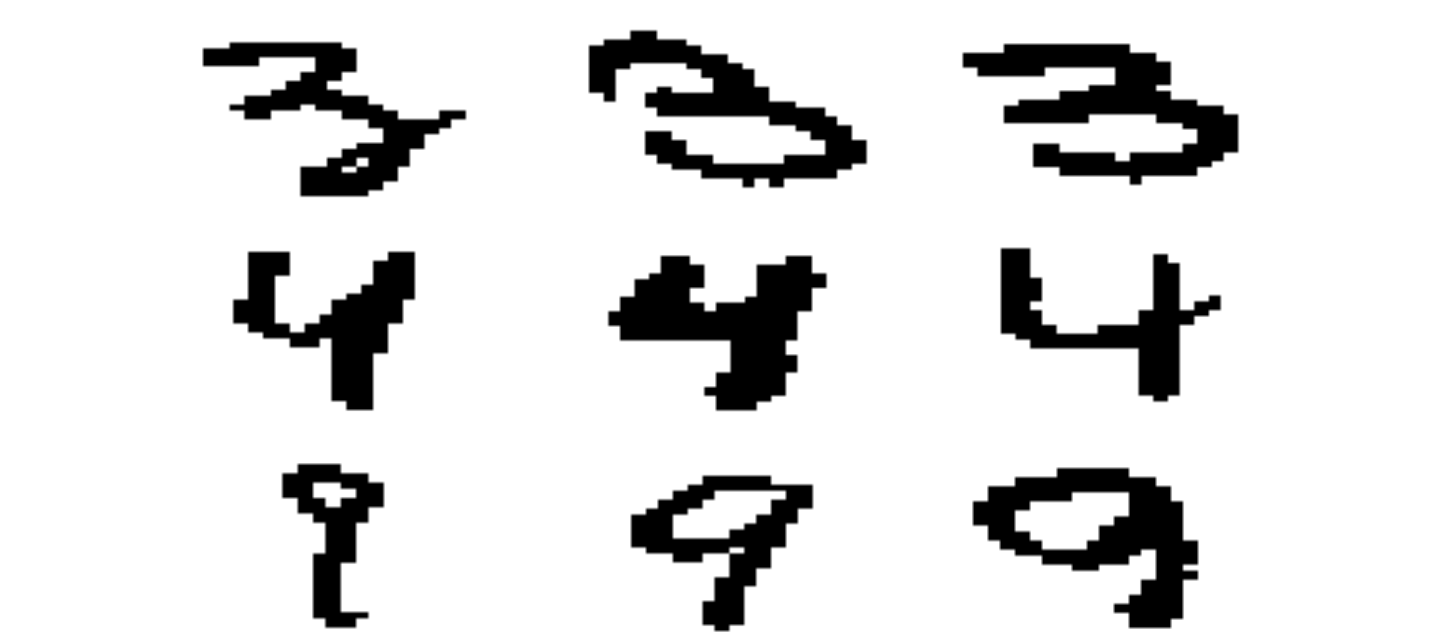}
\caption{Representation of dichotomized sample images: three digits 3 (first row), three digits 4 (second row) and three digits 9 (third row).} \label{xx}  
\end{figure}

The training sample size is $758$, $736$, $757$ and the test sample size is $252$, $246$, $252$ for digits 3, 4 and 9 respectively. The value $k=19$ (for the k-NN) and $100$ random projections are considered.
The confusion matrix of the binary classification in the test sample  are reported in Table \ref{table1}.

\begin{table}
\caption{Confusion matrices in the test samples of binary classifications.\label{table1}}
\centering
\begin{tabular}{|c|ccc|ccc|ccc|}
\hline
  & \multicolumn{9}{|c|}{ Labels}  \\ \cline{2-10} 
 &  & \textbf{3} & \textbf{4} & & \textbf{3}  & \textbf{9} & & \textbf{4} &  \textbf{9} \\  \hline
\multirow{2}{*}{Prediction} & \textbf{3} & 248 & 23 & \textbf{3}& 243 & 7& \textbf{4}& 230 & 17  \\
& \textbf{4} & 4 & 223& \textbf{9}& 9 & 245 & \textbf{9}& 16 & 235  \\ \cline{1-10} 
\end{tabular}
\end{table}
    
Therefore, the overall prediction errors are $5.4 \%$, $3.2 \%$ and $7.6 \%$ in the test sample for the classifications of digit pairs 3-4, 3-9 and 4-9 respectively.

\section{Conclusions}
We provide  some new statistical techniques to treat discrete high-dimensional data, based on some extensions of Heppes' theorem, in particular a quantitative version of Heppes' theorem that bounds total variation distance between two probability measures in the high-dimensional space by the sum of the total variation distances between their projections.

We show that, if we know in advance the supports of the distributions (which is the case, for instance, of multivariate Bernoulli distributions),
then it suffices to consider just one projection, provided that the direction is chosen to exclude a certain finite
or countable  set of bad directions, which are well determined.

Using the preceding results, we develop a new procedure for learning based on projections, and we adapt it to the case of discrete tomography, where we only have access to information related to the projections in a certain family of directions.

We also address the problem of testing for multivariate binary data, in particular for some well-known models like the Poisson-Binomial model. Our approach can be used for many other different problems for high-dimensional discrete data,
for example:
the Dirichlet-multinomial distribution, the multivariate hypergeometric distribution, the multivariate P\'olya--Eggenberger distributions, and the  negative multinomial distribution, among others (see \cite{johnson1997}).
We further analyze the case where we only have one realization of a multivariate Bernoulli distribution. 

Lastly, we perform a small simulation study and analyze three real-data examples. The results obtained  are quite encouraging. 

We conclude this article with a brief discussion of the potential extension of these ideas to continuous distributions.
As mentioned in the introduction, in order to distinguish between general continuous distributions, it is usually necessary to use
an infinite set of projections (see e.g.\ \cite{hamedani75}).
However, there is at least one case where finitely many projections will suffice. A Borel probability measure $P$ on $\RR^d$ is called \emph{elliptical} if its characteristic function has the form
\[
\phi_P(\xi)=e^{i\mu\cdot\xi}\psi(\xi^T \Sigma \xi)
\quad(\xi\in\mathbb{R}^d),
\]
where $\Sigma$ is a real positive semi-definite $d\times d$ matrix, 
$\mu$ is a vector in $\mathbb{R}^d$, and $\psi:[0,\infty)\to\mathbb{C}$ is a continuous function. Examples of elliptical distributions include multivariate versions of
Gaussian, Student, Cauchy, Bessel, uniform and logistic distributions. The following theorem  is a summary of 
results established in \cite{fraiman22}.

\begin{theorem}\label{T:CWelliptic}
Given $d\ge2$, there exists a set $\cL$ of $(d^2+d)/2$ lines in $\RR^d$ with the following property: if $P,Q$ are elliptical distributions on $\mathbb{R}^d$ such that $P_L=Q_L$ for all $L\in\cL$, then $P=Q$. No smaller set of lines will do.
\end{theorem}

Theorem~\ref{T:CWelliptic} leads  to the derivation statistical tests for equality of elliptical distributions. This program is carried out in \cite{fraiman22}.
An interesting avenue for further exploration, suggested by one of the referees, would be to combine the two cases already studied 
(discrete and elliptical) to develop tests for distributions that are discrete distributions with weak noise, for example, vector sums of discrete distributions and elliptical distributions with small variance.

\section*{Funding statement}
Fraiman and Moreno  were supported by grant FCE-1-2019-1-156054, Agencia Nacional de Investigaci\'on e Innovaci\'on, Uruguay. Ransford was supported by grants from NSERC and the Canada Research Chairs program.

\section*{Data availability statement}
All the data that we use are publicly available; the precise sources are cited in the article. The source code files in R are available from the authors upon request.

\bibliographystyle{rss}
\bibliography{biblio_ber}

\end{document}